\documentclass[11pt]{amsart}
\usepackage{amsmath,amsfonts,amssymb,amsthm}
\usepackage[alphabetic]{amsrefs}
\usepackage{hyperref}
\usepackage{graphicx,color}
\usepackage{enumerate}
\usepackage{pictexwd,dcpic}

\newtheorem{theorem}{Theorem}[section]
\newtheorem{lemma}[theorem]{Lemma}
\newtheorem{lem}[theorem]{Lemma}
\newtheorem{thm}[theorem]{Theorem}
\newtheorem{prop}[theorem]{Proposition}
\newtheorem{cor}[theorem]{Corollary}
\newtheorem{main}{Theorem}

\newtheorem{cmain}[main]{Corollary}

\theoremstyle{definition}
\newtheorem{defin}[theorem]{Definition}
\newtheorem{rem}[theorem]{Remark}

\newcommand{\C}{\mathbb{C}}
\newcommand{\Z}{\mathbb{Z}}
\newcommand{\R}{\mathbb{R}}
\newcommand{\Q}{\mathbb{Q}}
\newcommand{\N}{\mathbb{N}}
\newcommand{\lk}{\mathrm{lk}}

\newcommand{\ics}{invariant cocompact subcomplex}
\newcommand{\sics}{c.s.}
\newcommand{\eps}{\varepsilon}
\newcommand{\catz}{$\mathrm{CAT(0)}$ }

\begin{document}

\title[Tits Alternative for $2$-dimensional $\mathrm{CAT}(0)$ complexes]{Tits Alternative for $2$-dimensional $\mathrm{CAT}(0)$ complexes}

\author[D.~Osajda]{Damian Osajda$^{\dag}$}
\address{Instytut Matematyczny,
	Uniwersytet Wroc\l awski\\
	pl.\ Grun\-wal\-dzki 2/4,
	50--384 Wroc\-{\l}aw, Poland}
\address{Institute of Mathematics, Polish Academy of Sciences\\
	\'Sniadeckich 8, 00-656 War\-sza\-wa, Poland}
\email{dosaj@math.uni.wroc.pl}
\thanks{$\dag \ddag$ Partially supported by (Polish) Narodowe Centrum Nauki, UMO-2018/30/M/ST1/00668.}

\author[P.~Przytycki]{Piotr Przytycki$^{\ddag}$}
	\address{
Department of Mathematics and Statistics,
McGill University,
Burnside Hall,
805 Sherbrooke Street West,
Montreal, QC,
H3A 0B9, Canada}
\email{piotr.przytycki@mcgill.ca}
\thanks{$\ddag$ Partially supported by NSERC and AMS}

\begin{abstract}
We prove the Tits Alternative for groups acting on $2$-dimensional $\mathrm{CAT}(0)$ complexes with a bound on the order of the cell stabilisers.
\end{abstract}

\maketitle

\section{Introduction}

A \emph{triangle complex} $X$ is a 2-dimensional simplicial complex, with a following \emph{piecewise smooth Riemannian metric}.
Namely, we have a family of smooth Riemannian metrics $\sigma_T, \sigma_e$ on the triangles and edges such that the restriction of $\sigma_T$ to $e$ is $\sigma_e$ for each $e\subset T$. Riemannian metrics $\sigma_T, \sigma_e$ induce metrics (i.e.\ distance functions) on triangles and edges. We then equip~$X$ with the \emph{quotient pseudometric}~$d$ (see \cite[I.5.19]{BH}).
We assume that for each metric ball $B$, the simplices  
of $X$ intersecting $B$ have only finitely many isometry types. (Note that the only time we will apply it to $B$ of radius nonzero is in the proof of Remark~\ref{rem:translation}.)
Then $(X,d)$ is a complete length space, which can be
deduced from \cite[I.7.13 and I.5.20]{BH} using a bilipschitz map from each $B$ to a piecewise Euclidean complex. 
All group actions on~$X$ will be by simplicial isometries. 

We say that a group acts on a cell complex $X$ \emph{almost freely} if there is a bound on the order of the cell stabilisers. Note that an almost free action on a triangle complex is proper in the sense of \cite[I.8.2]{BH}. Furthermore, any subgroup of a group acting properly and cocompactly acts almost freely. 
\begin{main}
\label{thm:main}
Let $G$ be a finitely generated group acting almost freely on a $\mathrm{CAT}(0)$ triangle complex $X$. Then $G$ is virtually cyclic, or virtually~$\mathbb{Z}^2$, or contains a nonabelian free group.
\end{main}

By \cite[{II.7.5 and II.7.7(2)}]{BH} and Remarks~\ref{rem:semisimple} and~\ref{rem:translation}, if $G$ acts almost freely on a $\mathrm{CAT(0)}$ triangle complex with finitely many isometry types of simplices, then every sequence $G_1<G_2<\cdots$ of virtually abelian subgroups of $G$
stabilises. Consequently:

\begin{cmain}
	\label{cor:infgen}
If $X$ has finitely many isometry types of simplices, then Theorem~\ref{thm:main} 
holds also for $G$ infinitely generated. 
\end{cmain}

As explained in \cite[page 3]{OP}, one
cannot omit in Theorem~\ref{thm:main} the assumption on almost freeness.

Here are some examples of applications of 
Theorem~\ref{thm:main} 
to particular groups. The first result, which is a consequence of Corollary~\ref{cor:infgen}, was studied independently by Paul Tee. We are assuming that $G$ below acts freely instead of almost freely, since $A$ is torsion free \cite[Thm~B]{CD}.

\begin{cmain}
\label{cor:Artin}
Let $G$ be a
subgroup of a $2$-dimensional Artin group $A$ acting freely on the modified Deligne complex of $A$ (see \cite{CD}). Then $G$ is cyclic, $\mathbb{Z}^2$, the fundamental group of the Klein bottle, or contains a nonabelian free group.
\end{cmain}

The second result concerns the tame automorphism group $\mathrm{Tame}(\mathbf k^3)$ (see \cite{LP}).
In \cite[\S2 and \S5]{LP} we introduced a cell complex $\textbf{X}$ with an action of $\mathrm{Tame}(\mathbf k^3)$. We proved that $\textbf{X}$ is 
$\mathrm{CAT(0)}$ for $\mathbf k$ of characteristic $0$ \cite[Thm~A]{LP}. Some cells of $\textbf{X}$ are polygons
instead of triangles, but we can easily transform $\textbf{X}$ into a triangle complex
by subdividing.

\begin{cmain}
\label{cor:Tame}
Let $G$ be a finitely generated
subgroup of $\mathrm{Tame}(\mathbf k^3)$, with $\mathbf k$ of characteristic $0$. Suppose that $G$ acts almost freely on the cell complex~$\mathbf X$. Then $G$ is virtually cyclic, or virtually~$\mathbb{Z}^2$, or contains a nonabelian free group.
\end{cmain}

An ingredient in the proof of Theorem~\ref{thm:main} is the following characterisation of $\mathrm{CAT(0)}$ triangle complexes using a link condition. In \cite[Thm~7.1]{Babu} this was proved only for locally compact triangle complexes, and in \cite[II.5.2]{BH} only for piecewise Euclidean and piecewise hyperbolic triangle complexes.
\begin{main}
	\label{thm:localCAT(0)}
	A triangle complex $X$ is locally $\mathrm{CAT}(0)$ if and only if 
	\begin{enumerate}[(i)]
		\item the Gaussian curvature of $\sigma_T$ at any interior point of a triangle $T$ of~$X$ is $\leq 0$, and
		\item the sum of geodesic curvatures in any two distinct triangles of $X$ at any interior point of a common edge is $\leq 0$, and
		\item for each vertex $v$ of $X$, the girth of the link $\lk^X_v$ is $\geq 2\pi$.
	\end{enumerate}
\end{main}

\subsection*{Motivation and relation to other results.} The term \emph{Tits Alternative} usually refers to the property that all finitely generated subgroups are either virtually solvable or contain a nonabelian free group.
The name comes from the theorem of Tits \cite{Tits} who proved that every finitely generated linear group is either virtually
solvable or contains a nonabelian free group. It is widely believed (see e.g.\ \cite[Quest~2.8]{BestvinaProblem},\cite{BridsonICM},\cite[Quest~7.1]{BridsonProblem},\cite[Prob~12]{FHT},\cite[\S5]{Cap})
that all \emph{\catz groups} (groups acting \emph{geometrically}, that is, properly and cocompactly, on \catz spaces) satisfy the Tits Alternative. 
This was proved only in a limited number of cases: see \cite{NosVin,SagWis,CS,MP0,MPb,OP} and references therein.
%======other sources: \cite[Q.\ 5.4]{Duch}, \cite[Sec.\ 1]{CaMon}, 

Groups acting geometrically on $2$-dimensional \catz complexes were studied thoroughly by Ballmann and Brin in \cite{BB}, where they proved the Rank Rigidity Conjecture for such groups. They also proved that such groups are either virtually abelian, or contain a nonabelian free group (statements of this type are sometimes called the \emph{Weak Tits Alternative} \cite{SagWis}). However, the Tits Alternative for such groups has been open till our current work. (E.g.\ in \cite[Prob~12]{FHT} the question on the Tits Alternative is asked specifically in dimension $2$.) Just recently, together with Norin we were able to show in \cite{NOP}, among other results, that 
the groups in question do not contain infinite torsion subgroups. This property might be seen as the first step towards the Tits Alternative. In \cite{OP} we proved the Tits Alternative for the class of $2$-dimensional recurrent complexes. This class contains all $2$-dimensional Euclidean buildings, $2$-dimensional systolic complexes, as well as some complexes outside the \catz setting. 

Regarding Corollary~\ref{cor:Artin}, for right-angled Artin groups the Tits Alternative follows from the work of Sageev and Wise \cite{SagWis}.
In our previous work \cite{OP} we showed the Tits Alternative for a subclass
of $2$-dimensional Artin groups, containing all large-type Artin groups. Recently, in \cite{MPb} we
proved the Tits Alternative for $2$-dimensional Artin groups of hyperbolic type, and in \cite{MP0} we proved it for FC-type Artin groups.  An approach to the Tits Alternative for subgroups of $2$-dimensional Artin groups acting not
almost freely on the modified Deligne complex has been developed by Martin \cite{AM}. 

As for Corollary~\ref{cor:Tame}, Cantat proved that the group of birational transformations $\mathrm{Bir}(S)$, for a projective surface $S$ over any field $\textbf k$, satisfies the Tits Alternative \cite{Can}. Earlier, Lamy proved the Tits Alternative for the group of polynomial automorphisms $\mathrm{Aut}(\C^2)$ \cite{L1}, and the proof extends to any field $\textbf k$ of characteristic $0$. The same statement for $\mathrm{Aut}(\textbf k^3)$ seems at the moment out of reach. However, we believe that for $\mathrm{Tame}(\textbf k^3)\subsetneq \mathrm{Aut}(\textbf k^3)$, with $\textbf k$ of characteristic $0$, one could study the subgroups acting not
almost freely on the $\mathrm{CAT(0)}$ complex $\textbf{X}$ of \cite{LP} by generalising the methods of the current paper.

\subsection*{Organisation}
In Section~\ref{sec:CAT(0)} we prove Theorem~\ref{thm:localCAT(0)}. In Section~\ref{sec:cocompact} we recall the method of invariant cocompact subcomplexes from \cite{OP}, which allows us to reduce
Theorem~\ref{thm:main} 
to Proposition~\ref{prop:decompose} that assumes the existence of edges of degree $\geq 3$ in our complex $X$. Under this assumption we can exclude the cases of virtually cyclic or $\Z^2$ groups in Section~\ref{sec:notZ2}. In technical Section~\ref{sec:folding}, which we recommend to skip at the first reading, we arrange our complex $X$ to have no `unfoldable' links. We give criteria for finding `rank~1' elements, and consequently free subgroups, in Section~\ref{sec:rank1}. In the absence of `rank~1' elements, we obtain a particular rationality property of the complex~$X$ in Section~\ref{sec:extrationality}. In Section~\ref{sec:shear} we give new criteria for distinguishing the endpoints of certain piecewise geodesics. 
Together with a Poincar\'e recurrence argument this allows us to prove Proposition~\ref{prop:decompose} in Section~\ref{sec:free}.

\subsection*{Acknowledgements}
We thank Werner Ballmann, Martin Bridson, Florestan Brunck, Pierre-Emmanuel Caprace, Koji Fujiwara, St\'ephane Lamy, and Paul Tee for helpful remarks.

\section{Characterisation of $\mathrm{CAT}(0)$ triangle complexes} 
\label{sec:CAT(0)}
In this section we prove Theorem~\ref{thm:localCAT(0)}, which characterises $\mathrm{CAT}(0)$ triangle complexes. The following result is known under the name of the Cartan--Hadamard theorem.

\begin{thm}[{\cite[II.4.1(2)]{BH}}]
\label{thm:globalCAT(0)}
Let $X$ be a complete connected metric space. If $X$ is simply connected and locally $\mathrm{CAT}(0)$, then it is $\mathrm{CAT}(0)$.
\end{thm}

We also have the following consequence of \cite[II.1.7(4) and II.4.14(2)]{BH}.
\begin{thm}
\label{thm:localisometry}
Let $X$ be a complete 
$\mathrm{CAT}(0)$ space. A piecewise local geodesic in $X$ with Alexandrov angles $\pi$ at the breakpoints is a geodesic.
\end{thm}

Let $x$ be a point of a triangle complex $X$. Let $\lk^X_x$ be the metric graph that is the \emph{link} of $x$, as defined in \cite[page~176]{BB}. Namely, if $x$ is a vertex of $X$, then 
the vertices of $\lk^X_x$ correspond to the edges of $X$ containing $x$ and the edges of $\lk^X_x$ correspond to the triangles of $X$ containing $x$. The length of each edge is the angle in the corresponding triangle of $X$. Since we assumed that triangles containing~$x$ belong to only finitely many isometry classes of~$\sigma_T$, there are only finitely many possible edge lengths in a given~$\mathrm{lk}^X_x$. If $x$ lies in the interior of an edge $e$ of $X$, then $\lk^X_x$ has two vertices corresponding to the components of $e\setminus x$, and edges of length $\pi$ corresponding to the triangles of $X$ containing $e$. If $x$ lies in the interior of a triangle, then $\lk^X_x$ is a circle of length $2\pi$. We denote by $d_x^X$ (or, shortly, $d_x$) the length metric on~$\lk_x^X$. In \cite[\S2]{NOP} we explained how to identify $\lk^X_x$ with the completion of the space of directions at $x$ (see \cite[II.3.18]{BH}). Thus a local geodesic in~$X$ starting at~$x$ determines a point in $\mathrm{lk}^X_x$. The \emph{angle} at~$x$ between two such local geodesics is defined to be the distance between the two corresponding points in $\mathrm{lk}^X_x$ with respect to the metric $d_x$. As explained in \cite[\S2]{NOP}, if this angle is~$<\pi$, then it coincides with the Alexandrov angle, and if it is $\geq \pi$, then the Alexandrov angle equals $\pi$.

\begin{proof}[Proof of Theorem~\ref{thm:localCAT(0)}]
In the `only if' part, condition (i) follows from \cite[II.1A.6]{BH}. The proof of condition (ii) is identical to that in \cite[Thm~7.1]{Babu}, and the proof of condition (iii) was given in \cite[\S2]{NOP}. For the proof of the `if' part, suppose that a triangle complex $X$ satisfies conditions (i)--(iii). By condition~(i) and \cite[II.1A.6]{BH}, we have that $X$ is locally $\mathrm{CAT}(0)$ at any interior point of a triangle. 

Consider now an edge $e$ of $X$. Let $\mathrm{St}(e)$ be the union of all the closed triangles containing $e$. We will show that $\mathrm{St}(e)$ is $\mathrm{CAT}(0)$, which implies that $X$ is locally $\mathrm{CAT}(0)$ at any interior point $x$ of $e$, since the metrics on~$\mathrm{St}(e)$ and on $X$ coincide on a sufficiently small neighbourhood of $x$.
Let $Y\subset \mathrm{St}(e)$ be the union of the triangles $T$ for which there exists a point on~$e$ with positive geodesic curvature in $T$. By condition~(ii), there is at most one such triangle of given isometry type $T_0$ and given embedding $e\subset T_0$, so $Y$ has finitely many triangles. 
We denote this number by $m(e)$ for future reference. For each triangle $T$ of $\mathrm{St}(e)$ outside $Y$, denote $Y_T=Y\cup T$. By conditions~(i) and~(ii), and by \cite[Thm~7.1]{Babu}, we have that each $Y_T$ is locally $\mathrm{CAT}(0)$, hence $\mathrm{CAT}(0)$ by Theorem~\ref{thm:globalCAT(0)}. (Note that a geodesic in $Y_T$ might enter and exit a given triangle infinitely many times.) Furthermore, the inclusion $Y\subset Y_T$ is an isometric embedding, since points of $e$ have nonpositive geodesic curvature in $T$. By \cite[II.11.3]{BH}, the union~$\mathrm{St}(e)$ of~$Y_T$ is $\mathrm{CAT}(0)$, as desired.

Consider now a vertex $v$ of $X$. After possibly subdividing $X$, we can assume that in each triangle $T$ containing $v$, the local geodesic $\gamma_T$ starting at $v$ and bisecting the angle of $T$ at $v$ ends at the opposite side of $T$. We will prove that the union $\mathrm{St}(v)$ of all the closed triangles and edges containing $v$ is $\mathrm{CAT}(0)$.  This will imply that $X$ is locally $\mathrm{CAT}(0)$ at $v$, since the metrics on~$\mathrm{St}(v)$ and on $X$ coincide on a sufficiently small neighbourhood of $v$.

\smallskip

\noindent \textbf{Claim.} \emph{$\mathrm{St}(v)$ is geodesic, and there is $M>0$ such that each geodesic in~$\mathrm{St}(v)$ intersects the interiors of at most $M$ triangles.}

\smallskip

To justify the Claim, we employ the idea of a taut string \cite[I.7.20]{BH}. Let $\theta$ be the minimum of $\pi$ and the minimum length of an edge in $\lk_v^X$, and set $N=2+\frac{\pi}{\theta}$. Let $M=1+ N(1+\max_e m(e))$, where $m(e)$ is defined as above and the maximum is taken over all the edges $e$ of $\mathrm{St}(v)$ containing~$v$. For each such edge $e$, let $\mathrm{St}'(e)\subset \mathrm{St}(e)$ be the closure of the component containing $e$ of $\mathrm{St}(e)\setminus \bigcup_{T\subset \mathrm{St}(e)}\gamma_T$, for $\gamma_T$ as above. 
Since $\mathrm{St}'(e)$ is locally convex in $\mathrm{St}(e)$, it is $\mathrm{CAT}(0)$ \cite[II.4.14(1)]{BH}.

For $x,y\in \mathrm{St}(v)$, a \emph{string} between $x$ and $y$ is a sequence of edges $e_1,\ldots, e_n$ of $\mathrm{St}(v)$ containing $v$ and points $x_0=x,x_1,\ldots, x_n=y$ with $\mathrm{St}'(e_i)$ containing both $x_{i-1}$ and $x_i$. The \emph{length} of the string is the sum $\Sigma_{i=1}^nd_i(x_{i-1},x_i)$, where $d_i$ is the metric on $\mathrm{St}'(e_i)$. The distance between $x$ and $y$ in $\mathrm{St}(v)$ is the infimum of the lengths of strings between $x$ and $y$. A string is \emph{taut}, if $n\leq 2$ or 
\begin{itemize}
\item
for each $0\leq i\leq n,$ the point $x_i$ is distinct from $v$, and
\item
for each $0<i<n,$ the point $x_i$ belongs to one of the $\gamma_T$ defined above, and 
\item
for each $0<i<n,$ the concatenation of geodesics $x_{i-1}x_i$ in $\mathrm{St}'(e_i)$ and $x_{i}x_{i+1}$ in $\mathrm{St}'(e_{i+1})$ is a geodesic in $\mathrm{St}'(e_i)\cup \mathrm{St}'(e_{i+1})$. 
\end{itemize}

We now justify that for each string $(e_i),(x_i)$ between $x$ and $y$ we can find a taut string between $x$ and $y$ whose length does not exceed the length of $(e_i),(x_i)$. Indeed, by discarding some $x_i$, we can first assume that consecutive 
$e_i$ are distinct and so for each $0<i<n$ the point $x_i$ belongs to one of the $\gamma_T$. Since $\gamma_T$ are compact, there is a choice of $x'_i$ in the same~$\gamma_T$ as $x_i$, minimising the length of the string $(e_i),(x'_i)$. Then the concatenation of geodesics $x'_{i-1}x'_i$ in~$\mathrm{St}'(e_i)$ and $x'_{i}x'_{i+1}$ in~$\mathrm{St}'(e_{i+1})$ is a geodesic in~$\mathrm{St}'(e_i)\cup \mathrm{St}'(e_{i+1})$. Finally, if an $x'_{i-1}$ equals~$v$, and $x'_{i}\neq x_n$, then for $x'_i\in \gamma_T$, the subpath of $\gamma_T$ from $x'_{i-1}$ to $x'_{i}$ is a geodesic in both $\mathrm{St}'(e_i)$ and $\mathrm{St}'(e_{i+1})$, and consequently we can discard $x'_{i}$ and $e_i$ from the string. Repeating the argument we arrive at $i-1=n-1$ or $i-1=n$. Analogously, we obtain $i-1=1$ or $i-1=0$, and so $n\leq 2$.

We will now show that a taut string satisfies $n\leq N$. We can assume $n>2$, and so none of $x_i$ equals $v$. For $0<i<n,$ let $\theta_i$ be the Alexandrov angle at $x_{i}$ in $\mathrm{St}'(e_{i+1})$ between the geodesics $x_{i}x_{i+1}$ and $x_{i}v$.  
The concatenation of geodesics $x_{i-1}x_i$ in $\mathrm{St}'(e_i)$ and $x_{i}x_{i+1}$ in $\mathrm{St}'(e_{i+1})$ is a geodesic in $\mathrm{St}'(e_i)\cup\mathrm{St}'(e_{i+1})$, which is $\mathrm{CAT}(0)$ by \cite[II.11.3]{BH}. Consequently, by the definition of $\theta$ we have $\theta_i\geq \theta_{i-1}+\theta$. Thus $\pi\geq \theta_{n-1}\geq (n-2)\theta +\theta_1$, and so $n-2\leq \frac{\pi}{\theta}=N-2$.

Since there are finitely many isometry types of simplices in $\mathrm{St}(v)$, and each taut string satisfies $n\leq N$, the distance between $x$ and $y$ in $\mathrm{St}(v)$ is realised by the length of some taut string $(e_i),(x_i)$. Then the concatenation of all the geodesics $x_{i-1}x_i$ in $\mathrm{St}'(e_i)$ is a geodesic between $x$ and $y$, proving that $\mathrm{St}(v)$ is geodesic. Furthermore, for any geodesic $\gamma$ from $x$ to $y$ in~$\mathrm{St}(v)$, one can choose points on $\gamma$ forming a taut string. Since any taut string satisfies $n\leq N$, and we have that $\gamma$ intersects the interiors of at most $1+n(1+\max_e m(e))$ triangles, the Claim follows.

\smallskip

Returning to the proof of Theorem~\ref{thm:localCAT(0)}, we follow the scheme in \cite[Thm~7.1]{Babu} to find a sequence of $\mathrm{CAT}(0)$ spaces Gromov--Hausdorff converging to $\mathrm{St}(v)$. Namely, realise each (isometry type of a) triangle $T$ of $\mathrm{St}(v)$ as $T\subset \R^2$, with metric induced from some smooth Riemannian metric of nonpositive Gaussian curvature on $\R^2$ defined in a neighbourhood $U$ of~$T$. Denote by $e,f$ the edges of $T$ containing~$v$, and by $g$ the remaining edge of~$T$. Let $l$ denote the length of $e$. For each $n>0$, we decompose $e$ into paths $a^1\cdot a^2 \cdots a^n$ of length $\frac{l}{n}$, and we define $\kappa^k$ to be the integral of the geodesic curvature along~$a^k$.
Let $e_n$ be the piecewise geodesic in $U$ that starts at~$v$ tangent to~$e$, has $n$ locally geodesic pieces of length~$\frac{l}{n}$, and exterior angle at the $k$-th breakpoint that equals $\kappa^k$. For $n$ sufficiently large the path $e_n$ exists, and they $C^1$-converge to~$e$ as $n$ tends to $\infty$. We define paths~$f_n$ analogously. We define $g_n$ to be any piecewise geodesics joining the endpoints of~$e_n$ and~$f_n$ and $C^1$-converging to $g$.
This gives us a triangle $T_n\subset U$ bounded by $e_n,f_n$ and~$g_n$, whose boundary is piecewise geodesic (one can pass to a union of triangles with locally geodesic boundary by subdividing). Furthermore, we have a map $T_n\to T$ whose restriction to $e_n,f_n$ preserves length and which is bilipschitz with 
the bilipschitz constant converging to $1$ as $n$ tends to $\infty$. Glueing various $T_n$ along the sides corresponding to the ones of $T$ that we glued to form $\mathrm{St}(v)$ yields a triangle complex that we call $\mathrm{St}(v)_n$. Then $\mathrm{St}(v)$ is a Gromov--Hausdorff limit of~$\mathrm{St}(v)_n$.
Note that since $\mathrm{St}(v)$ satisfied conditions (i)--(iii), we have that $\mathrm{St}(v)_n$ satisfies conditions (i)--(iii).
Since $\mathrm{St}(v)_n$ has locally geodesic edges, and is geodesic by the Claim (applied to $\mathrm{St}(v)_n$ instead of $\mathrm{St}(v)$), the same proof as for \cite[Thm~7.1, case 1]{Babu} shows that $\mathrm{St}(v)_n$ is locally $\mathrm{CAT}(0)$, hence $\mathrm{CAT}(0)$ by Theorem~\ref{thm:globalCAT(0)}. Consequently, by \cite[II.3.10]{BH} we have that $\mathrm{St}(v)$ is $\mathrm{CAT}(0)$.
\end{proof}

We have the following immediate consequence of Theorem~\ref{thm:localCAT(0)}.

\begin{cor}
\label{cor:subcomplex}
If $X$ is a triangle complex that is locally $\mathrm{CAT}(0)$, then all of its subcomplexes are locally $\mathrm{CAT}(0)$.
\end{cor}

Note that while we did not apply the Claim to $\mathrm{St}(v)$ in the proof of Theorem~\ref{thm:localCAT(0)}, it will be used in the following remarks.

\begin{rem}
\label{rem:semisimple} Suppose that $X$ is a $\mathrm{CAT}(0)$ triangle complex with finitely many isometry types of simplices. Then the constant $M$ in the Claim does not depend on $v$. As in the `$\Psi_n\Rightarrow \Phi_n$' part of the proof of \cite[I.7.28]{BH}, we obtain that for each $l$ there is $M'>0$ such that each geodesic in $X$ of length $\leq l$ intersects the interiors of at most $M'$ simplices. Using this in the place of \cite[Lem~1]{B} in the proof of \cite[Lem 2 and Thm~A]{B}, we obtain that every simplicial isometry $g$ of $X$ is \emph{semisimple}: it fixes a point, or is \emph{loxodromic}, meaning that there is a geodesic line $\omega$ in $X$ (called an \emph{axis}) 
such that $g$ preserves $\omega$ and acts on it as a nontrivial translation. 
\end{rem}

\begin{rem} 
\label{rem:translation}
Suppose that $X$ is a $\mathrm{CAT}(0)$ triangle complex with finitely many isometry types of simplices. Then the set of translation lengths of simplicial isometries of $X$ is a discrete subset of $[0,\infty)$, which is proved using the Claim exactly as \cite[Prop]{B}. Similarly we obtain the following:

Let $X$ be a $\mathrm{CAT}(0)$ triangle complex
with a subcomplex $Y$ on which some group of simplicial isometries of $X$ acts coboundedly. Since any metric ball in $X$ intersects finitely many isometry types of simplices, we have that each bounded neighbourhood of $Y$ intersects finitely many isometry types of simplices.
Then for each simplicial isometry~$g$ of $X,$ the set $\inf_{y\in Y}d(y,gy)$ attains its infimum, which we denote $|g|_Y$. Moreover, the set of $|g|_Y$ over all 
simplicial isometries $g$ of~$X$ is a discrete subset of $[0,\infty)$.
\end{rem}

\section{$G$-cocompact subcomplexes}
\label{sec:cocompact}

Let $X$ be a simplicial complex with an action of a group~$G$. We say that a subcomplex $Z\subset X$ is
an \emph{\ics \ with respect to~$G$} (shortly \emph{$G$-\sics}) if $Z$ is $G$-invariant, and the quotient $Z/G$
is compact. Note that a $G$-\sics\ is not required to be connected.

A $2$-dimensional simplicial complex is \emph{essential} if every edge has degree at least $2$, and none of connected components is a single vertex. An essential simplicial complex is \emph{thick} if it has an edge of degree at least $3$.

A \emph{disc diagram} $D$ is a compact contractible simplicial complex with a fixed embedding in $\R^2$. Its \emph{boundary path} is the attaching map of
the cell at $\infty$. If $X$ is a simplicial complex, \emph{a disc diagram in $X$} is a nondegenerate simplicial map $\varphi \colon D\to X$, and its
\emph{boundary path} is the composition of the boundary path of $D$ and $\varphi$. We say that $\varphi$ is \emph{reduced} if it maps triangles sharing an edge to two distinct triangles. 
By \cite[Rem~3.6]{OP}, for each contractible closed edge-path $\alpha$ in a simplicial complex $X$, there is a reduced disc diagram in $X$ with boundary path $\alpha$.

A group $G$ acts on a simplicial complex $X$ \emph{without inversions} if for any $g\in G$ stabilising a simplex $\sigma$ of $X$ we have that $g$ fixes $\sigma$ pointwise. More generally, we say that $G$ acts \emph{without weak inversions} if for each vertex $v$ of~$X$ there is no $g\in G$ sending $v$ to a distinct vertex in a common edge.

The first ingredient in our proof of 
Theorem~\ref{thm:main} 
is the following earlier result.
\begin{prop}[{\cite[Prop 3.7]{OP}}]
\label{cor:OP}
Let $G$ be a finitely generated group acting almost freely and without inversions on a simply connected $2$-dimensional simplicial complex $X$ that contains no simplicial $2$-spheres. If $X$ contains no thick {$G$-\sics}, then
$G$ is virtually cyclic, or virtually $\mathbb{Z}^2$, or contains a nonabelian free group.
\end{prop} 

The second ingredient in the proof of 
Theorem~\ref{thm:main} 
is the following, the proof of which will occupy the present article.

\begin{prop}
\label{prop:decompose}
Let $G$ be a group acting almost freely and without weak inversions
on a $\mathrm{CAT}(0)$ triangle complex $X$ that is an increasing union of connected essential $G$-\sics\ 
If $X$ contains an edge of degree $\geq 3$, then $G$ contains a nonabelian free group.
\end{prop}

We now show how 
Theorem~\ref{thm:main} 
follows from these two ingredients.

\begin{proof}[Proof of Theorem~\ref{thm:main}]
By passing to a subdivision (see \cite[Lem~2.1]{NOP}), we can assume that $G$ acts without weak inversions. By Proposition~\ref{cor:OP}, we can assume that $X$ contains a thick $G$-\sics\ $Z_1$. We will prove that $G$ contains a nonabelian free group.
By passing to a connected component of $Z_1$ and its stabiliser $G'$ in $G$ (which is finitely generated, since it acts properly and cocompactly on a connected complex), we can assume that $Z_1$ is connected.
If $Z_1$ contains a closed edge-path that is not contractible in $Z_1$, repeatedly attaching to~$Z_1$ the images of 
reduced disc diagrams and their $G$-translates, we obtain an increasing sequence $Z_1\subset Z_2\subset \cdots$
of connected essential {$G$-\sics} such that their union $X'$ is simply connected. By Corollary~\ref{cor:subcomplex} we have that $X'$ is locally $\mathrm{CAT}(0)$, and so $X'$ is $\mathrm{CAT}(0)$ by Theorem~\ref{thm:globalCAT(0)}.
It remains to apply Proposition~\ref{prop:decompose}.
\end{proof}

\section{Not virtually cyclic or $\Z^2$}
\label{sec:notZ2}

The first step of the proof of Proposition~\ref{prop:decompose} is the following. 

\begin{lem}
\label{lem:noteasy}
Let $G$ be a group acting almost freely 
on a $\mathrm{CAT}(0)$ triangle complex $X$.
If $X$ contains a subcomplex $Z$ that is a connected thick $G$-\sics,
then
\begin{enumerate}[(i)]
\item
$G$ is not virtually cyclic, and
\item
$G$ is not virtually $\Z^2$.
\end{enumerate}
\end{lem}

In the proof of Lemma~\ref{lem:noteasy} we will need the following vocabulary. Let $X$ be a triangle complex.
We say that a ray $\gamma\colon [0,\infty)\to X$ or a path $\gamma\colon [0,1]\to X$ \emph{starts} (resp.\ \emph{ends}) in a simplex $\sigma$, if for some $\eps>0$ the points $\gamma(0,\eps)$ (resp.\ $\gamma(1-\eps,1)$) all lie in the interior of $\sigma$. If $\gamma(0)$ (resp.\ $\gamma(1)$) lies in the interior of an edge $e$, then $\gamma$ \emph{starts} (resp.\ \emph{ends}) \emph{perpendicularly to $e$} if the angle at $\gamma(0)$ (resp.\ $\gamma(1)$) between $\gamma$ and $e$ is $\frac{\pi}{2}$.

We will also need the following, which generalises an argument in the proof of \cite[Thm~A]{MPa}. 

\begin{lemma}
\label{lem:flattorus}
Let $A$ be a group isomorphic to $\Z^2$ acting freely on a $\mathrm{CAT}(0)$ triangle complex $W$ with finitely many isometry types of simplices. Then there is an isometrically embedded $A$-cocompact subcomplex in $W$ isometric to the Euclidean plane.
\end{lemma}
\begin{proof}
Since $W$ has finitely many isometry types of simplices, by Remark~\ref{rem:semisimple} all elements of~$A$ act as loxodromic isometries on $W$. By \cite[II.7.20(1)]{BH}, we have
$\mathrm{Min}(A) = Y \times \R^n$ 
with~$A$ preserving the product structure and acting trivially on $Y$. By \cite[II.7.20(2)]{BH}, we have $n\leq 2$, but since~$A$ acts freely by simplicial isometries, we have $n=2$, and so $Y$ is a point, as desired. 
\end{proof} 

\begin{proof}[Proof of Lemma~\ref{lem:noteasy}]
Let $e$ be an edge of $Z$ of degree $\geq 3$ and let $x$ be a point in the interior of $e$. Let $b_1,b_2,b_3,$ be geodesics
starting at~$x$ perpendicularly to~$e$ contained in distinct triangles $T_1,T_2,T_3$. For each $i=1,2,3,$ the set of starting directions at points in $\partial T_i$ of geodesics intersecting $b_i$ at angle $<\frac{\pi}{6}$ has positive Liouville measure (see \cite[\S3]{BB}). 
Let $S$ denote the union of all the open edges in the links $\lk_y^Z$ for all $y\in  Z^1\setminus Z^0$, with the (infinite) Liouville measure. 
We say that $(\xi_j)_j\in S^\Z$ with $\xi_j\in \lk_{y_j}^Z$ \emph{determines} a locally geodesic oriented line~$\gamma$ in $Z\setminus Z^0$ transverse to $Z^1$
if $\gamma$ intersects~$Z^1$ exactly at points $y_j$ in directions $\xi_j$, in that order.
Since $G$ acts on~$Z$ properly and cocompactly, the set of 
$(\xi_j)_j\in S^\Z$ that determine locally geodesic oriented lines projects to a full measure subset in each coordinate~$S$ (see \cite[\S3]{BB}, which relies on \cite[Chap~6]{CFH}). 
Consequently, for $i=1,2,3,$ there exists a locally geodesic ray $\gamma_i$ in $Z$ starting at an interior point $x_i$ of $b_i$ at angle $<\frac{\pi}{6}$ from~$b_i$, disjoint from~$Z^0$ and transverse to $Z^1$. Let $a_i=xx_i\subset b_i$. See Figure~\ref{f:0}. 

\begin{figure}[h!]
	\centering
	\includegraphics[width=0.66\textwidth]{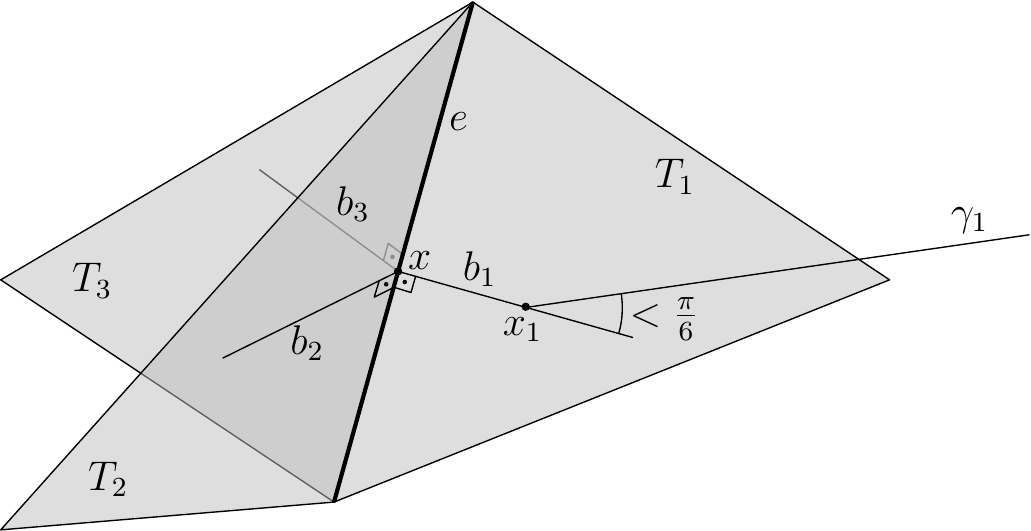}
	\caption{The geodesic ray $\gamma_1$}
	\label{f:0}
\end{figure}

Since $X$ is $\mathrm{CAT}(0)$, by Theorem~\ref{thm:localisometry} we have that $\gamma_i$ are geodesic rays in~$X$ and $a^{-1}_i\cdot a_j$ are geodesics in $X$. Since each $\gamma^{-1}_i\cdot a^{-1}_i\cdot a_j\cdot \gamma_j$ is a piecewise geodesic with angles $>\frac{5\pi}{6}$ at the two breakpoints, by \cite[II.9.3]{BH} the rays $\gamma_i,\gamma_j$ are not asymptotic and they determine points at distance $>\frac{2\pi}{3}$ in the Tits boundary of $X$.
In particular, $Z$ cannot be quasi-isometric to $\R$, since it contains three pairwise non-asymptotic geodesic rays. This proves~(i). 

For (ii), assume for contradiction that $G$ is virtually $\Z^2$ generated by elements $g,h$. Let $\alpha, \beta$ be edge-paths in $Z$ connecting a basepoint $y\in Z^0$ to $gy,hy$, respectively. Then the concatenation $\alpha \cdot g\beta \cdot h\alpha^{-1}\cdot \beta^{-1}$ is a closed edge-path, and since $X$ is simply connected, there is a reduced disc diagram $D\to X$ with that boundary path. Let $Z'\subset X$ be the connected thick $G$-\sics\ obtained from $Z$ by adding the translates under $G$ of the image of~$D$. The complex $Z'$ is locally $\mathrm{CAT}(0)$ by Corollary~\ref{cor:subcomplex}. Let $\widetilde Z'\to Z'$ be the universal cover of $Z'$, which is $\mathrm{CAT}(0)$ by Theorem~\ref{thm:globalCAT(0)}. The action of $G$ on~$Z'$ lifts to an almost free action of a group $\widetilde G$ on~$\widetilde Z'$ fitting into the short exact sequence $\pi_1 Z'\to \widetilde G\to G$. Since $D\to Z'$ lifts to $D\to \widetilde Z'$, we have commuting $\tilde g, \tilde h\in \widetilde G$ mapping to $g,h\in G$, and hence generating a subgroup $A< \widetilde G$ isomorphic to $\Z^2$.

Since $A$ acts almost freely and is torsion-free, we have that it acts freely on~$\widetilde Z'$. By Lemma~\ref{lem:flattorus} applied with $W=\widetilde Z'$, there is an isometrically embedded 
$A$-cocompact subcomplex $E\subset \widetilde Z'$ isometric to the Euclidean plane. We now justify that the composition $\phi\colon E\subset \widetilde Z'\to Z'\subset X$ is an isometric embedding. 

Indeed, for two triangles $T,T'$ of $E$ containing a common edge~$e'$,
the sum of the geodesic curvatures in $\phi(T)$ and $\phi(T')$ at any point of $\phi(e')$ equals~$0$, and so the geodesic curvature at $\phi(e')$ in any triangle of $X$ distinct from $\phi(T),\phi(T')$ is nonpositive. Consequently, $\phi$ is a local isometric embedding at~$e'$. 
Furthermore, since $E$ is isometric to the Euclidean plane, for any geodesic~$\gamma$ in~$E$ passing through a vertex $v$, the angle (see \S\ref{sec:CAT(0)}) between the incoming and outgoing directions of $\gamma$ at $v$ equals $\pi$. Since the map that $\phi$ induces between $\lk_v^E$ and $\lk_{\phi(v)}^X$ is locally injective, the angle between the incoming and outgoing directions of $\phi(\gamma)$ at $\phi(v)$ equals~$\pi$ as well. By Theorem~\ref{thm:localisometry}, we have that $\phi(\gamma)$ is a geodesic. Thus $\phi$ is an isometric embedding, as desired.

Since $A$ is of finite index in $G$, we have that $A$ acts cocompactly on $Z'$. Consequently, the geodesic rays $\gamma_i$ from the proof of part (i) are at bounded distance from $\phi(E)$ in $Z'$. Since $X$ is $\mathrm{CAT}(0)$, we obtain that each $\gamma_i$ is asymptotic to a geodesic ray in $\phi(E)$ and these three rays are pairwise at angle $>\frac{2\pi}{3}$ in $\phi(E)$, which is a contradiction. 
\end{proof}

\section{Folding}
\label{sec:folding}
This section is devoted to a technical reduction of Proposition~\ref{prop:decompose} to the case where the vertex links of $X$ are not `unfoldable'.

By a \emph{graph} we mean a (possibly infinite) metric graph with finitely many possible edge lengths.
A closed edge-path embedded in a graph~$\Lambda$ is a \emph{cycle of~$\Lambda$}.
An edge-path $I$ in $\Lambda$ that is embedded, except possibly at the endpoints, is a \emph{segment of~$\Lambda$} if the endpoints of $I$ have degree $\geq 3$
in $\Lambda$, but every internal vertex of $I$ has degree $2$.

 \begin{defin} 
Let $S$ be a set with an equivalence relation $\sim$ each of whose equivalence classes has size $\geq 2$.
A~graph is a \emph{$\sim$-clover} if it is obtained from the disjoint union of intervals $S\times [0,\pi]$ by identifying all the points in $S\times 0$ to one point called the \emph{basepoint} and identifying each $s\times \pi$ with $s'\times \pi$ for $s\sim s'$. See Figure~\ref{f:1}.
A~graph is a \emph{clover} if it is a $\sim$-clover for some $S,\sim$.  
A~graph $\Gamma$ is \emph{unfoldable} at a vertex $y$ if $\Gamma$ is a wedge $\Gamma_1\vee\Gamma_2$ at $y$ of a cycle~$\Gamma_1$ of length~$2\pi$ and a 
clover $\Gamma_2$ with basepoint $y$. (In particular, $\Gamma$ is also a clover.)

\begin{figure}[h!]
	\centering
	\includegraphics[width=0.36\textwidth]{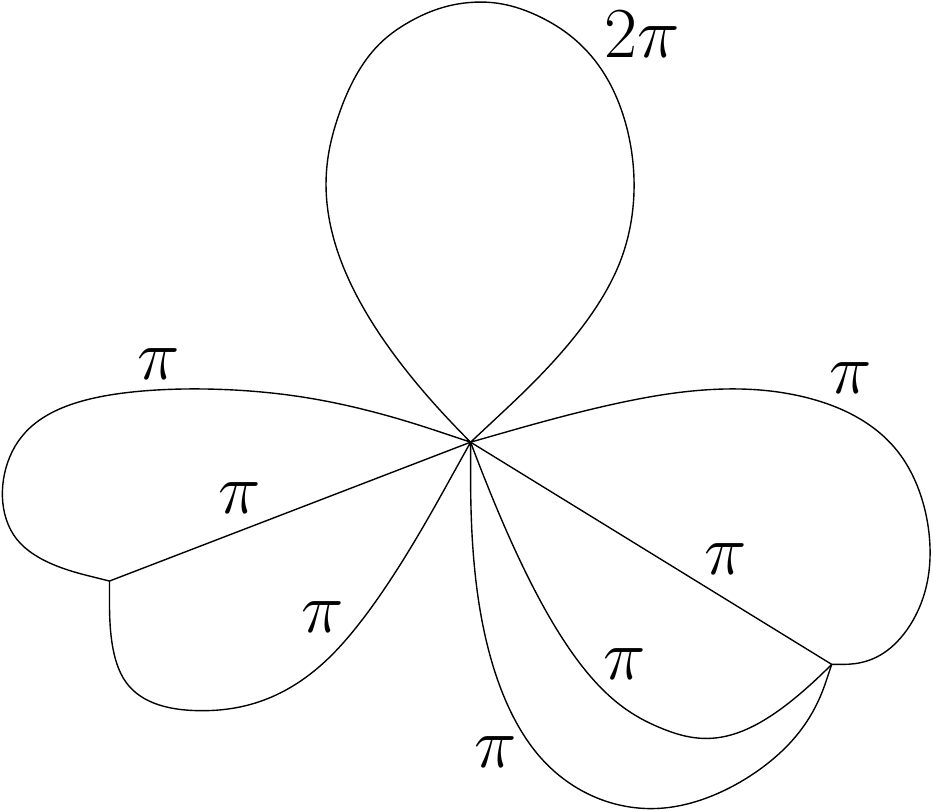}
	\caption{A clover}
	\label{f:1}
\end{figure}
 
Suppose that we have a triangle complex $X'$ and a vertex $w'$ contained in distinct edges $e_1=w'v_1,e_2=w'v_2$ of the same length. Suppose that $\mathrm{lk}^{X'}_{v_1}$ is a circle of length $2\pi$ and $\mathrm{lk}^{X'}_{v_2}$ is a 
clover with basepoint corresponding to $e_2$
Then the quotient map $p'\colon X'\to X$ with $X$ obtained from $X'$ by identifying $v_1$ with $v_2$ and $e_1$ with~$e_2$ 
is called a \emph{folding}. (Note that $X$ might not be a simplicial complex, but in this article we will be using only the inverse operation to folding which does result in a simplicial complex.) 

Conversely, suppose that $X$ is a triangle complex with a vertex $v$ whose link $\Gamma_1\vee\Gamma_2$ is unfoldable at a point $y$ corresponding to an edge $vw$. 
Then, up to an isometry, there exists a unique triangle complex $X'$ and a folding $p'\colon X'\to X$ identifying edges $w'v_1,w'v_2$ to $wv$ and such that the links of~$v_i$ in~$X'$ map isometrically to the graphs $\Gamma_i$ in the link of~$v$. See Figure~\ref{f:4}. We call $p'$ the \emph{folding over $\Gamma_1$} (since it is uniquely determined by $\Gamma_1$). 

\begin{figure}[h!]
	\centering
	\includegraphics[width=0.68\textwidth]{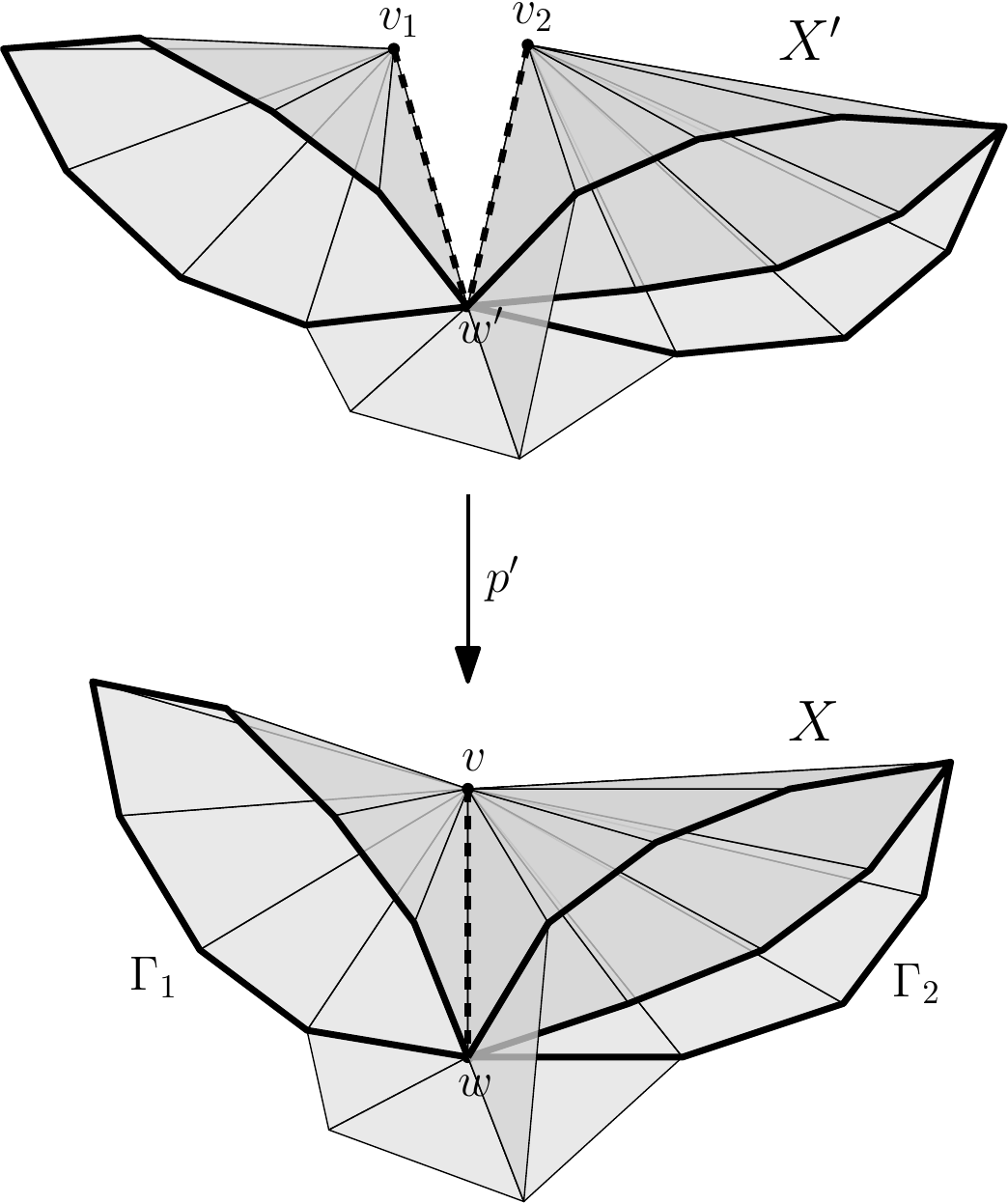}
	\caption{A folding $p'\colon X'\to X$ over $\Gamma_1$. Subgraphs corresponding to the links $\Gamma_1$ and $\Gamma_2$ and their preimages are thickened, the edge $vw$ and its preimage are dashed.}
	\label{f:4}
\end{figure}

Suppose now that $p'\colon X'\to X,\hat p'\colon \hat X'\to X$ are foldings over $\Gamma_1\neq\hat \Gamma_1$. Suppose that $v\neq \hat w$ and $\hat v\neq w$.  
We have that $\hat\Gamma_1$ lifts to a link $\lk_{v'}^{X'}$, which is again unfoldable,
except when $\lk_{v'}^{X'}=\hat \Gamma_1$. In that exceptional case, we have $\Gamma_2=\hat \Gamma_1$ and $p'=\hat p'$, and we set $p''=\mathrm{id}$.
Otherwise, let $p''\colon X''\to X'$ be the folding over~$\hat \Gamma_1$. We call $p''\circ p'$ the \emph{folding over $\Gamma_1,\hat\Gamma_1$}. Note that the folding over $\Gamma_1,\hat \Gamma_1$ coincides with the folding over $\hat \Gamma_1,\Gamma_1$. 

Analogously, given a 
finite family of foldings $p'^\lambda\colon X'^\lambda\to X$ over $\Gamma^\lambda_1$ (where $\lambda$ is an index) with all $v^\lambda$ distinct from all $w^{\lambda}$, the \emph{folding over $\{\Gamma^\lambda_1\}$}
is the composition of foldings over the lifts of $\Gamma^\lambda_1$, which does not depend on the order. For a countable family of such foldings $p'^\lambda\colon X'^\lambda\to X$, the \emph{folding over $\{\Gamma^\lambda_1\}$} is
the inverse limit of the foldings over the finite subsets of $\{\Gamma^\lambda_1\}$.
\end{defin}

\begin{lemma} 
\label{rem:folding}
Let $p_{\mathcal F}\colon X_{\mathcal F}\to X$ be the folding over a (finite or countable) family ${\mathcal F}=\{\Gamma^\lambda_1\}$. 
\begin{enumerate}[(i)]
\item
The map $p_{\mathcal F}$ is a homotopy equivalence.
\item
If $X$ is locally $\mathrm{CAT}(0)$, then $X_{\mathcal F}$ is locally $\mathrm{CAT}(0)$. 
\item 
If $X$ is essential, then $X_{\mathcal F}$ is essential. 
\end{enumerate}
\end{lemma}
\begin{proof}
For part (i), note that if for some indices $\lambda,\mu$, we have $v^\lambda=v^\mu$, then $w^\lambda=w^\mu$, since the point $y^\lambda$ in $\lk_{v^\lambda}^X$ does not depend on~$\Gamma^\lambda_1$. Consequently, distinct edges $v^\lambda w^\lambda$ might intersect only along $w^\lambda$. Thus their union $V\subset X$ is a forest, and so the quotient map $q\colon X\to X^*$ collapsing each component of $V$ into a point is a homotopy equivalence. Similarly, the subcomplex $V_{\mathcal F}=p_{\mathcal F}^{-1}(V)\subset X_{\mathcal F}$ is a forest, and so the quotient map $q_{\mathcal F}\colon X_{\mathcal F}\to X^*$ collapsing each component of $V_{\mathcal F}
$ into a point is also a homotopy equivalence. Thus the identity $q_{\mathcal F} =q\circ p_{\mathcal F}$ implies part (i).

Part (ii) follows from the fact that the maps that $p_{\mathcal F}$ induces between the links of $X_{\mathcal F}$ and $X$ are locally injective, and from Theorem~\ref{thm:localCAT(0)}.

The map $p_{\mathcal F}$ is a local isometry at the open edges outside $V_{\mathcal F}$. Thus for part (iii) it suffices to justify that the links of the vertices in each $p_{\mathcal F}^{-1}(v^\lambda)$ have no leaves. But each such link is a cycle $\Gamma^\mu_1$ (for $v^\mu=v^\lambda$) or a clover, as desired.
\end{proof}

\begin{prop}
\label{prop:unfolding}
Let $G$ be a group acting almost freely and without weak inversions
on a $\mathrm{CAT}(0)$ triangle complex $X$ that is an increasing union of connected essential $G$-\sics\
Then $G$ acts almost freely and without weak inversions on a $\mathrm{CAT}(0)$ triangle complex $X'$ that is an increasing union of connected essential $G$-\sics\
and none of whose links are unfoldable. 

Furthermore, if $X$ contains an edge of degree $\geq 3$, then $X'$ contains an edge of degree $\geq 3$ or $G$ contains a nonabelian free group.
\end{prop}
 \begin{proof}
We fix an increasing sequence $Z_k\subset X$ of connected essential $G$-\sics\  exhausting~$X$.
A \emph{multifolding} $(X',(Z'_k),p')$ is a:
\begin{enumerate}[(i)]
\item
$\mathrm{CAT}(0)$ triangle complex $X'$ with an action of $G$, 
\item
a sequence $(Z'_k)$ of essential 
$G$-\sics\ exhausting $X'$, and 
\item
a $G$-equivariant simplicial map $p'\colon X'\to X$ that 
\begin{itemize}
\item
maps bijectively the set of triangles of
each $Z'_k$ to the set of triangles of $Z_k$, and 
\item
whose restriction  $Z'_k\to Z_k$ is a homotopy equivalence. 
\end{itemize}
\end{enumerate}
We introduce a partial order $\leq$ on the set of multifoldings, writing $(X',(Z'_k),p')\leq(X'',(Z''_k),p'')$ (or, shortly, $X'\leq X''$) if there is a $G$-equivariant simplicial map $r\colon X''\to X'$ satisfying $p''=p'\circ r$. Multifoldings $X',X''$ are \emph{equivalent} if $X'\leq X''$ and $X''\leq X'$. Let $\mathcal X$ be the set of equivalence classes of multifoldings.

We claim that every chain of elements $X'^{\lambda}$ in $\mathcal X$ (where $\lambda$ is an index) has an upper bound. Indeed, denote by $p'_k$ the restriction of $p'$ to $Z'_k$ and write $(Z'_k,p'_k)\leq(Z''_k,p''_k)$ whenever there is a $G$-equivariant simplicial map $r\colon Z''_k\to Z'_k$ satisfying $p''_k=p'_k\circ r$. For each $k$, since $G$ acts properly and cocompactly on $Z'^{\lambda}_k$, by the first bullet
we have that $(Z'^{\lambda}_k, p'^\lambda_k)$ can take on only finitely values up to the appropriate equivalence. Thus there exists a largest element among the $(Z'^{\lambda}_k, p'^\lambda_k)$, which we call $Z'^{\infty}_k$. Furthermore, since 
$(Z'_{k+1},p'_{k+1})\leq(Z''_{k+1},p''_{k+1})$ implies $(Z'_{k},p'_k)\leq(Z''_{k},p''_k)$, 
we have natural injective maps $Z'^{\infty}_k\to Z'^{\infty}_{k+1}$. Let $X'^{\infty}$ be their direct limit, equipped with the limit map~$p'^{\infty}$ to~$X$. Since each $Z'^{\infty}_k$ is locally $\mathrm{CAT}(0)$ (Corollary~\ref{cor:subcomplex}),
we have that $X'^{\infty}$ is locally $\mathrm{CAT}(0)$ (Theorem~\ref{thm:localCAT(0)}). 

To prove that $X'^{\infty}$ is an upper bound for our chain in $\mathcal X$, by Theorem~\ref{thm:globalCAT(0)} it remains to prove that $X'^{\infty}$ is simply connected. Let $\alpha$ be a closed edge-path in the $1$-skeleton of $X'^{\infty}$, and fix $k$ such that $\alpha$ lies in $Z'^{\infty}_k$. Fix~$\lambda$ with $Z'^{\lambda}_k=Z'^{\infty}_k$ and keep the notation $\alpha$ for its copy in $Z'^{\lambda}_k$. Since $X'^{\lambda}$ is simply connected, there is a  disc diagram $D\to X'^{\lambda}$ with boundary path $\alpha$. Fix $l$ such that the image of $D$ is contained in $Z'^{\lambda}_l$. Since, by the second bullet, 
the induced map $\pi_1Z'^{\infty}_l\to \pi_1Z'^{\lambda}_l$ is an isomorphism, we have that $\alpha$ is trivial in $\pi_1Z'^{\infty}_l$ and hence in $\pi_1X'^{\infty}$.
Consequently, by the Kuratowski--Zorn lemma, there is a maximal element $X'\in \mathcal X$.

We now prove that none of the links of $X'$ are unfoldable. Otherwise, suppose that  $X'$ has a vertex $v$ whose link $\Gamma=\Gamma_1\vee\Gamma_2$ is unfoldable at a point corresponding to an edge $vw$ of $X'$. 
Let $\mathcal F=\{g\Gamma_1\}$, for $g\in G$. Since $G$ acts without weak inversions, we have $gv\neq hw$, for all $g,h\in G$. 
Thus we can define the folding over $\mathcal F$, which we denote by $X_{\mathcal F}\to X'$. The action of~$G$ on~$X'$ lifts to an action of $G$ on $X_{\mathcal F}$. For each essential $G$-\sics\ $Z'\subset X'$, let $Z_{\mathcal F}\subset X_{\mathcal F}$ be the closure of the union of all the open triangles of $X_{\mathcal F}$ mapping into $Z'$. Note that if $Z'$ does not contain $v$, or if $Z'$ contains~$v$, but $\mathrm{lk}^{Z'}_v$ is contained in $\Gamma_1$ or $\Gamma_2$ (in the latter case $\mathrm{lk}^{Z'}_v\subset \bigcap_{g\in \mathrm{Stab}(v)}g\Gamma_2$), then $Z_{\mathcal F}\to Z'$ is an isometry. Otherwise $\mathrm{lk}^{Z'}_v$ contains edges lying on both the cycle $\Gamma_1$ and the clover~$\Gamma_2$. Since $Z'$ is essential, the link $\mathrm{lk}^{Z'}_v$ is the wedge of $\Gamma_1$ and a clover, so it is unfoldable. Then $Z_{\mathcal F}\to Z'$ is the folding over the family $\{g\Gamma_1\}$, for $g\in G$.
By Lemma~\ref{rem:folding}(i,iii) we have that $Z_{\mathcal F}$ is essential and $Z_{\mathcal F}\to Z'$ is a homotopy equivalence. By Lemma~\ref{rem:folding}(i,ii) we have that $X_{\mathcal F}$ is simply connected and locally $\mathrm{CAT}(0)$, and hence $\mathrm{CAT}(0)$ by Theorem~\ref{thm:globalCAT(0)}.
Consequently, we have $X_{\mathcal F}\in \mathcal X$ and $X_{\mathcal F}>X'$, contradicting the maximality of $X'$. Thus none of the links of $X'$ are unfoldable.

For the last assertion, note that, by Lemma~\ref{lem:noteasy} applied to $X$, we have that $G$ is neither virtually cyclic, nor virtually $\mathbb{Z}^2$. Moreover, $G$ is finitely generated, since it acts properly and cocompactly on $Z_1$, which is connected.
Consequently, if $X'$ does not have edges of degree $3$, then by Proposition~\ref{cor:OP} we have that $G$ contains a nonabelian free group, as desired.
\end{proof}

\section{Criteria for rank~1 elements}
\label{sec:rank1}

In this section, we give criteria for finding `rank~$1$' elements, and consequently free subgroups in $G$.

\begin{defin}
\label{def:rank}
Let $\gamma$ be a geodesic line in a $\mathrm{CAT}(0)$ triangle complex $X$. We say that $\gamma$ is \emph{curved} if
$\gamma$ passes through a vertex $v$ and its incoming and outgoing directions at $v$ 
are at angle $>\pi$. 
\end{defin}

\begin{lem}
\label{prop:contracting}
Let $g$ be a loxodromic isometry of a $\mathrm{CAT}(0)$ triangle complex~$X$ with a curved axis~$\gamma$. Then there exists $M$ such that the projection to $\gamma$ of each closed metric ball in $X$ disjoint from $\gamma$ has diameter $\leq M$.
\end{lem}

\begin{proof} Suppose that $\gamma$ passes through a vertex $v$ with incoming and outgoing directions at angle
$>\pi+\kappa,$ for some $\frac{\pi}{2}>\kappa>0$. Let $R$ be the translation length of~$g$.
We will prove that $M=R\lceil\frac{2\pi}{\kappa}\rceil$ satisfies the lemma. Otherwise, let $x,y$ be points in a closed metric ball disjoint from~$\gamma$, such that the projections $x',y'$ of $x,y$ to $\gamma$ are at distance $>M$. 

There are at least $n=\lceil\frac{2\pi}{\kappa}\rceil$ translates of $v$ under $\langle g \rangle$ on $x'y'$ distinct from $x',y'$. We denote these translates by $v'_1,\ldots, v'_n$, in order in which they appear on $x'y'$. By the continuity of the projection map, there are points $v_1,\ldots,v_n$ lying on the geodesic $xy$ in that order such that each $v_i'$ is the projection of~$v_i$ to~$\gamma$. We additionally denote $v_0=x, v_0'=x', v_{n+1}=y,v'_{n+1}=y'$. For $0\leq i\leq n$,
let $\beta_i$ denote the geodesic quadrilateral $v_iv_{i+1}v'_{i+1}v'_iv_i$. The sum of the four Alexandrov angles of each~$\beta_i$ is $\leq 2\pi$ \cite[II.2.11]{BH}, so the sum of all the Alexandrov angles of all $\beta_i$ is $\leq (n+1)2\pi$. 

On the other hand, for $0\leq i<n$, the sum of the Alexandrov angles of~$\beta_i$ and~$\beta_{i+1}$ at $v_{i+1}$ is $\geq \pi$. We will now prove that the sum of the Alexandrov angles of~$\beta_i$ and~$\beta_{i+1}$ at $v'_{i+1}$ is $>\pi+\kappa$. Indeed, if one of them is not equal to the angle in the usual sense (see $\S2$), then it equals $\pi$. However, since $v'_{i+1}$ is the projection of $v_{i+1}$, the second Alexandrov angle is $\geq \frac{\pi}{2}$, so their sum is $\geq\frac{3\pi}{2}$, as desired. Consequently, we have $n(\pi+\pi+\kappa)< (n+1)2\pi$, and so $n\kappa< 2\pi$, which is a contradiction.
\end{proof}

\begin{lem}
\label{lem:Rtree}
Let $G$ be a group acting almost freely 
on a $\mathrm{CAT}(0)$ triangle complex $X$ with a fixed point $\xi$ in the visual boundary of~$X$. Suppose that there is a curved axis $\gamma$ for some $g\in G$ with one of the limit points $\xi$. Then $G$ is virtually cyclic.
\end{lem}
\begin{proof}
Consider the space of geodesic rays $\rho\colon [0,\infty)\to X$ representing $\xi$, with the pseudometric $d(\rho_1,\rho_2)=\inf_{t_1,t_2}d(\rho_1(t_1),\rho_2(t_2))$. Identifying $\rho_1$ with $\rho_2$ for $d(\rho_1,\rho_2)=0,$ we obtain a metric space whose metric completion~$X_\xi$ is $\mathrm{CAT}(0)$ \cite[Prop~2.8]{L}. Since $X$ has geometric dimension $\leq 2$ (see \cite{K}), by \cite[Rem after Cor~4.4]{C}, we have that $X_\xi$ has geometric dimension $\leq 1$ (more precisely, as we learned from Pierre-Emmanuel Caprace, for any $\rho$ representing $\xi$ the space $X_\xi\times \R$ embeds isometrically in the pointed ultralimit of $(X,\rho(n))_n$, which has geometric dimension $\leq 2$ by \cite[Lem~11.1]{LY}). Since $G$ fixes $\xi$, the action of $G$ on~$X$ induces an action of $G$ on $X_\xi$. 

We will now justify that a complete $\mathrm{CAT}(0)$ space~$X_\xi$ of geometric dimension $\leq 1$ is an $\R$-tree, which we learned also from Pierre-Emmanuel Caprace. For a geodesic triangle $xyz$ in $X_\xi$, let $x'$ be the projection of $x$ to the geodesic~$yz$. If $x'\neq x,y$, then the direction of the geodesic $x'x$ in~$\lk_{x'}^{X_\xi}$ is distinct from that of the geodesic $x'y$.
Thus the geodesic $xy$ must pass through~$x'$, since otherwise mapping it to~$\lk_{x'}^{X_\xi}$ would give a path between distinct points of a discrete set. Analogously, the geodesic $xz$ passes through~$x'$, and so $xyz$ is $0$-thin, justifying that $X_\xi$ is an $\R$-tree. 

Suppose first that there is $h\in G$ acting loxodromically on $X_\xi$. Let $M$ be the constant given by Lemma~\ref{prop:contracting} for $\gamma$. Let $\rho$ be a ray in $\gamma$ representing~$\xi$. Since $h$ acts loxodromically on $X_\xi$, after possibly replacing $h$ by its power, we can assume $d(\rho,h\rho)>M$. Assume without loss of generality that the Busemann function (see \cite[II.8.17 and II.8.20]{BH}) satisfies $B_\xi(\rho(0))\leq B_\xi(h\rho(0))$. 
Then for each $t\geq 0$, the projection $p(t)$ of $\rho(t)$ to $h\gamma$ is contained in $h\rho$, and for $D=d(\rho(0),h\rho(0))$ we have $d(\rho(t),p(t))\leq D$. Pick $k\in \N$ with $kM>2D$. Then by the triangle inequality we have $d(p(0),p(2kM))\geq 2kM-2D>kM$. Consequently, there is $0\leq n<k$ such that $d(p(2nM),p(2(n+1)M))>M$. Thus the closed ball of radius $M$ centred at $\rho((2n+1)M)$ is disjoint from~$h\gamma$ and contains points $\rho(2nM),\rho(2(n+1)M)$ whose projections to $h\gamma$ are at distance $>M$. This contradicts the choice of $M$.

Consequently, $G$ has a global fixed point in $X_\xi$ (which might not be represented by a geodesic ray, but be a point added in the completion). Thus there is $D>0$ such that for any $\eps>0$ there is a geodesic ray $\rho'$ representing~$\xi$ at distance $\leq D$ from $\rho$ and satisfying $d(\rho',g\rho')<\eps$ for each $g\in G$. Consider the homomorphism $\psi\colon G\to \R$ defined by $\psi(g)=B_\xi(gx)-B_\xi(x)$ for any $x\in X$. We will now justify that $\psi$ has discrete image. Otherwise, for any $\eps>0$ and $\rho'$ as above there is $t>0$ such that $d(\rho'(t),g\rho'(t))<2\eps$, but 
$g$ does not fix a point of $X$. This contradicts Remark~\ref{rem:translation} applied with $Y$ containing the $D$-neighbourhood of $\gamma$.

Let $K$ be the kernel of $\psi$. Arguing as in the previous paragraph, we obtain that every $g\in K$ fixes a point of $X$. By \cite[Thm~1.1(i)]{NOP}, every finitely generated subgroup of $K$ fixes a point of $X$. Since $K$ acts almost freely, we have that $K$ is finite and so $G$ is virtually cyclic. 
\end{proof}

\begin{prop}
\label{cor:contracting}
Let $G$ be a group acting almost freely 
on a $\mathrm{CAT}(0)$ triangle complex $X$. 
Assume that $G$ contains a loxodromic element~$g$ with a curved axis $\gamma$. Then $G$ is virtually cyclic or contains a nonabelian free group.
\end{prop}

\begin{proof}
By Lemma~\ref{prop:contracting}, $g$ is rank~$1$ in the sense of \cite[Def~5.1]{BF}. By Lemma~\ref{lem:Rtree} we can assume that $G$ does not have a finite index subgroup fixing a limit point of $\gamma$.
Then there is $f\in G$ with $\gamma$ and $f\gamma$ having disjoint limit point pairs. Consequently, by \cite[Prop~5.9]{BF}, for some $n$ the elements $g^n$ and $fg^nf^{-1}$ generate a nonabelian free group.
\end{proof}

\section{Extrationality}
\label{sec:extrationality}

The main result of this section will be Proposition~\ref{prop:rank1}, where we will show that in the absence of unfoldable vertices and curved axes, the complex~$X$ enjoys a particularly strong rationality property of angles, which we call extrationality.

\begin{defin}
\label{def:patch}
The \emph{branching locus} $E$ of a triangle complex $X$ is the subcomplex of $X$ that is the union of all the closed edges of degree $\geq 3$.
A \emph{patch} of $X$ is a maximal connected subspace $P$ of $X\setminus E$ such that $P\setminus X^0$ is connected; see Figure~\ref{f:2}. If $X$ is simply connected, then by Van Kampen's theorem $P$ is a planar surface, and so we can choose an orientation on~$P$.

We equip $P\setminus X^0$ with the length metric induced from $X\setminus X^0$ (see \cite[I.3.24]{BH}). Let $\overline P$ denote the completion of $P\setminus X^0$, which admits an obvious embedding of $P$. Furthermore, $\overline P$ admits an obvious triangle complex structure and a simplicial map  $\overline P\to X$
fitting the following commutative diagram.
\[
\begindc{\commdiag}[20]
\obj(12,1)[a']{$\overline{P}$}
\obj(12,1)[a]{}
\obj(35,1)[b]{$X$}
\obj(35,1)[b']{}
\obj(35,16)[c]{$P$}
\obj(35,17)[c']{}
\mor{a}{b'}{}
\mor{c}{b}{}
\mor{c}{a}{}
\enddc
\]
Note that $\overline{P}$ is a connected surface with boundary, which we denote $\partial P$.
\end{defin}

\begin{figure}[h!]
	\centering
	\includegraphics[width=0.8\textwidth]{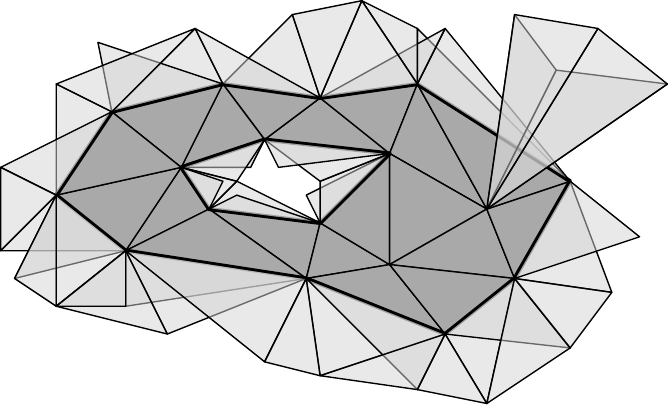}
	\caption{A patch (in dark grey) 
	}
	\label{f:2}
\end{figure}

\begin{defin}
A triangle complex $X$ is \emph{piecewise Euclidean} if all its triangles are geodesic Euclidean triangles.
A piecewise Euclidean triangle complex $X$ is \emph{rational} if for any vertex $v$ of $X$ all cycles and segments (see \S\ref{sec:folding}) in the link of $v$ have lengths commensurable with $\pi$. In particular, the angle at~$v$ between any edges of the branching locus $E$ is then commensurable with $\pi$.
A rational triangle complex~$X$ is \emph{extrational}, if 
\begin{itemize}
\item for any vertex $v$ of $X$ with a component $C$ of $\mathrm{lk}^X_v$ a circle, we have that the length of~$C$ is $2\pi$, and 
\item each homomorphism $\psi$ defined below is trivial.
\end{itemize}
We define $\psi=\psi(P)$ for each patch $P$ of $X$. Consider the chain complex $C_*(\overline P,\partial P)$ consisting of those singular chains that are affine w.r.t.\ the affine structure on $\overline P$ induced by the piecewise Euclidean metric. Note that the affine structure on $\overline P$ has singularities at the points $x$ of $\partial P$ with $\lk_x^{\overline P}$ of length $\neq \pi$, and so we require our affine chains to be disjoint from such~$x$ except possibly at the vertices. For each $x\in \overline P$ choose (not necessarily continuously) a direction $\xi_x\in \mathrm{lk}^{\overline P}_x$ at $x$, with the only restriction that for $x\in \partial P$, the direction $\xi_x$ corresponds to one of the edges in~$\partial P$ containing~$x$. For an affine singular $1$-simplex $\sigma\to \overline P$ with endpoints $x$ and $y$, let $\psi(\sigma)\in \R/\pi\Q$ be the oriented angle 
between~$\xi_x$ and~$\sigma$ at $x$ minus the oriented angle between~$\xi_y$ and $\sigma$ at $y$. Note that since $X$ was rational, this equals $0$ mod $\pi\Q$ for $\sigma$ in~$\partial P$, and so we obtain a homomorphism $\psi\colon C_1(\overline P,\partial P)\to \R/\pi\Q$. Note that the restriction of $\psi$ to $Z_1(\overline P,\partial P)$ does not depend on the choice of the $\xi_x$. Furthermore, for each affine singular $2$-simplex $\tau$ we have $\psi(\partial \tau)=\pm\pi=0$ mod $\pi\Q$, and so $\psi$ descends to a homomorphism $\psi\colon H_1(\overline P,\partial P)\to \R/\pi\Q$. It is not hard to check that our $H_1(\overline P,\partial P)$ coincides with the usual first (singular) homology group.
\end{defin}

We will need the following variant of \cite[Lem~7.4]{BB} that was implicit in the proof of \cite[Prop 3.4]{NOP}.

\begin{lem}
\label{lem:2npoints}
Let $G$ be a group acting almost freely
on a $\mathrm{CAT}(0)$ triangle complex $X$ that is an increasing union of 
essential $G$-\sics\
Furthermore, assume that there is a vertex $v$ of $X$ with points $\xi_i,\eta_i\in \mathrm{lk}^X_v$, for $i=1,\ldots, n,$ such that 
\begin{itemize}
\item
 $d^X_v(\xi_i,\eta_i)=\pi$ for $i=1,\ldots, n,$ and 
 \item
 $d^X_v(\eta_i,\xi_{i+1})\geq \pi$ for $i=1,\ldots, n-1,$ and 
 \item
 $d^X_v(\eta_n, \xi_1)>\pi$.
\end{itemize}
Then $G$ contains a loxodromic element $g$ with a curved axis.
\end{lem}

\begin{proof} Let $Z\subset X$ be an essential $G$-\sics\ containing $v$, such that $\mathrm{lk}^{Z}_v$ contains $\xi_i,\eta_i$, with
$d^{Z}_v(\xi_i,\eta_i)=\pi,$ for $i=1,\ldots, n$. 
By \cite[Lem~5.4]{NOP} (which was stated in terms of the compact quotient but has the same proof for proper and cocompact actions), for any $\eps>0$ there is a path $\omega=\omega_1\cdots \omega_{6n}$ in $Z$ such that
\begin{itemize}
\item
paths $\omega_j$ are local geodesics in $Z$, and
\item
there are $g_0=\mathrm{id}, g_1,\ldots, g_n=g\in G$ such that paths $\omega_{3i+1}$ start at $g_iv$, paths $\omega_{3i}$ end at $g_{i}v$, and except for that $\omega$ is disjoint from the vertex set~$Z^0$ and transverse to $Z^1$, and
\item
the starting direction of $\omega_{3i+1}$ is at distance $<\frac{\eps}{2}$ to $g_i\xi_{i+1}$ in~$\mathrm{lk}^{Z}_{g_iv}$, and the ending direction of $\omega_{3i}$ is at distance $<\frac{\eps}{2}$ to $g_i\eta_i$, and
\item
at the remaining breakpoints, $\omega_{j}$ and $\omega_{j+1}$ are at angle $>\pi-\eps$ (and $\leq \pi$ since outside $Z^0$).
\end{itemize}

Since $\omega_j$ are disjoint from $Z^0$ and transverse to $Z^1$, by Theorem~\ref{thm:localisometry} they are geodesics in $X$. The last two bullets hold in $X$ as well. In particular, 
at all the breakpoints $\omega_{j}$ and $\omega_{j+1}$ are at angle $>\pi-\eps$. By \cite[Lem~2.5]{BB}, the geodesic $\gamma$ in $X$ with the same endpoints as $\omega$ starts and ends in directions at distance $<(6n-1)\eps$ to $\xi_1,g\eta_n$. Consequently, for $\eps$ sufficiently small, by Theorem~\ref{thm:localisometry} we have that $\bigcup_{l\in \Z}g^l\gamma$ is a curved axis for $g$. 
\end{proof}

\begin{prop}
\label{prop:rank1}
Let $G$ be a group acting almost freely
on a $\mathrm{CAT}(0)$ triangle complex $X$ that is an increasing union of essential $G$-\sics\
and none of whose links are unfoldable. 
If $X$ is not $G$-equivariantly isometric (by a possibly non-simplicial isometry) to a piecewise Euclidean triangle complex~$X'$ or 
$X$ is isometric to such $X'$ but $X'$ is not extrational, then $G$ is virtually cyclic or contains a nonabelian free group.
\end{prop}
\begin{proof}
To prove that $G$ is virtually cyclic or contains a nonabelian free group in each case we will show the existence of a curved axis in $X$ (or in a different $\mathrm{CAT}(0)$ triangle complex $\overline X$) for an element of $G$, since then the proposition follows from Proposition~\ref{cor:contracting}.

Assume first that $X$ is not $G$-equivariantly isometric to a piecewise Euclidean triangle complex $X'$. Then by \cite[Prop~2.11]{BB} there is 
\begin{enumerate}[(i)]
\item
a point in the interior of a triangle of $X$ with negative Gaussian curvature, or 
\item
a point in the interior of an edge of $X$ with negative sum of geodesic curvatures of some two incident triangles, or
\item
a vertex $v$ of $X$ with $\mathrm{lk}^X_v$ a circle of length $>2\pi$. 
\end{enumerate}

In case (iii), or, more generally, if $\mathrm{lk}^X_v$ has a component $C$ that is a circle of length $>2\pi$, let $\xi_1,\eta_1$ be points at distance $\pi$ in $C$, and let $\eta_2,\xi_2$ be their antipodal points. Applying 
Lemma~\ref{lem:2npoints} with $n=2$, we obtain a curved axis. In cases~(i) and~(ii), by \cite[Lem~5.5]{NOP}, there is a $\mathrm{CAT}(0)$ triangle complex~$\overline X$, obtained from~$X$ by a $G$-equivariant subdivision and a $G$-equivariant replacement of the smooth Riemannian metrics, 
with a vertex $u\in \overline X$ whose $\lk^{\overline X}_u$ is either 
\begin{itemize}
\item
a circle of length $>2\pi$, or 
\item
a graph obtained from a family of disjoint circles $C_1,C_2,\ldots$ of length~$2\pi$ by glueing them along a nontrivial arc $b$ of length $<\pi$. 
\end{itemize}

The first bullet brings us to case (iii). In the case of the second bullet, let $\xi_1,\xi_2\in C_1\setminus b$ and $\eta_1,\eta_2\in C_2\setminus b$ be points at distance $\frac{\pi}{2}$ from the endpoints of $b$, with $d^{\overline X}_u(\xi_1,\eta_1)=d^{\overline X}_u(\xi_2,\eta_2)=\pi$. Applying 
Lemma~\ref{lem:2npoints} with $n=2$, we obtain a curved axis in $\overline X$, as desired. 

Thus without loss of generality we can assume that $X$ is a piecewise Euclidean triangle complex. If $X$ is not rational, then by 
\cite[Prop~7.7]{BB}, applied to an essential $G$-\sics, there is a closed locally injective edge-path $\beta$ in some $\mathrm{lk}^X_v$ whose length is not commensurable with $\pi$. In particular, by \cite[Lem~6.1(iii)]{BB}, there are points $\xi,\eta$ in $\mathrm{lk}^X_v$ at distance $>\pi+\delta,$ for some $\delta>0$. (One could apply \cite[Cor~1.7]{NOP} to find such $\xi,\eta$ in $\beta$, but it does not simplify the argument.) Let $\beta_-$ (resp.\ $\beta_+$) be the shortest path from $\xi$ (resp.\ $\eta$) to $\beta$. Since the length of $\beta$ is not commensurable with~$\pi$, there is a path $\beta_-\beta_0\beta_+$ with $\beta_0$ factoring through the universal cover of $\beta$ whose length equals $(2n-1)\pi+\delta'$ for some $n\in \N$ and $0\leq \delta'\leq \delta$. Choosing $\xi_1=\xi,\eta_1,\xi_2,\ldots, \eta_{n}$ as consecutive points at distance~$\pi$ along that path, we have $d^{X}_v(\xi_1,\eta_n)>\pi$. Applying 
Lemma~\ref{lem:2npoints}, we obtain a curved axis. 

Finally, if $X$ is not extrational, let $P$ be a patch of $X$ with nontrivial $\psi=\psi(P)$. Since $P$ is planar, there is an element in $H_1(\overline P,\partial P)$ represented by a piecewise affine path $\alpha$ in $\overline P$ with endpoints in $\partial P$ and $\psi(\alpha)\neq 0$. Let $\alpha$ be shortest among such paths, which exists since $\mathrm{Stab}(P)$ acts cocompactly on $\overline P$. Note that then $\alpha$ does not intersect~$\partial P$ except at its endpoints, since otherwise we could decompose it into two shorter paths, with $\psi$ nontrivial on at least one of them. 
Thus the image of~$\alpha$ in $X$ (for which we keep the same notation)
%and our $H_1(\overline P,\partial P)$ coincides with the singular first homology group, there is 
is a  local 
geodesic in $X$ that intersects the branching locus $E$ exactly at its endpoints $x,x'$. Let $e$ (resp.~$e'$) be the segment of $\mathrm{lk}^X_x$ (resp.\ $\mathrm{lk}^X_{x'}$) containing the point corresponding to the direction of $\alpha$. By the shortness condition, we have that $e$ (resp.~$e'$) has endpoints at distance $\geq \frac{\pi}{2}$ from $x$ (resp.\ $x'$), and so is of length $\geq \pi$. They cannot both have length~$\pi$, since then we would have $\psi(\alpha)=0$, so assume without loss of generality that the length $l$ of $e$ is $>\pi$. 

If $l> 2\pi$, then it is easy to find points $\eta_1,\xi_1,\xi_2,\eta_2$ lying on $l$ in that order and satisfying the hypothesis of Lemma~\ref{lem:2npoints} with $n=2$. If $2\pi> l>\pi$ or $l=2\pi$ and the endpoints of~$e$ are distinct, then the construction of such points is given in the proof of \cite[Lem~7.6]{BB}. It remains to consider the case where $l=2\pi$ and where both endpoints of $e$ are equal to a vertex $y$. Let $\Gamma_2$ be the graph obtained from $\mathrm{lk}^X_x$ by removing $e$. If $\Gamma_2$ contains a point~$z$ at distance $>\pi$ from~$y$, then it is easy to find points $\xi_1,\eta_1,\xi_2,\eta_2$ on a geodesic from $z$ to the midpoint of $e$ satisfying the hypothesis of Lemma~\ref{lem:2npoints} with $n=2$. Otherwise, $\Gamma_2$ is a clover, contradicting the assumption that $\mathrm{lk}^X_x$ is not unfoldable.
\end{proof}

\section{Sheared geodesics}
\label{sec:shear}

In this section we prove that the following piecewise geodesics have distinct endpoints. We will use some vocabulary from \S\ref{sec:notZ2}.

\begin{defin}
\label{def:sheared}
A \emph{sheared geodesic} in a piecewise Euclidean triangle complex $X$ is a concatenation $\gamma_1\cdot\gamma_2\cdots \gamma_{2k-1}\cdot\gamma_{2k}$ of geodesics such that (see Figure~\ref{f:5}):
\begin{itemize}
\item
for $i=1,\ldots, k,$ the (possibly trivial) geodesic $\gamma_{2i}$ lies in the interior of an edge $e_i$ of~$X$, and
\item
for $i=1,\ldots, k-1,$ the geodesic $\gamma_{2i-1}$ ends and the geodesic $\gamma_{2i+1}$ starts perpendicularly to $e_i$ in triangles of $X$ that are distinct, and the geodesic $\gamma_{2k-1}$ ends perpendicularly to $e_k$ in a triangle.
\end{itemize}
\end{defin}

\begin{figure}[h!]
	\centering
	\includegraphics[width=0.99\textwidth]{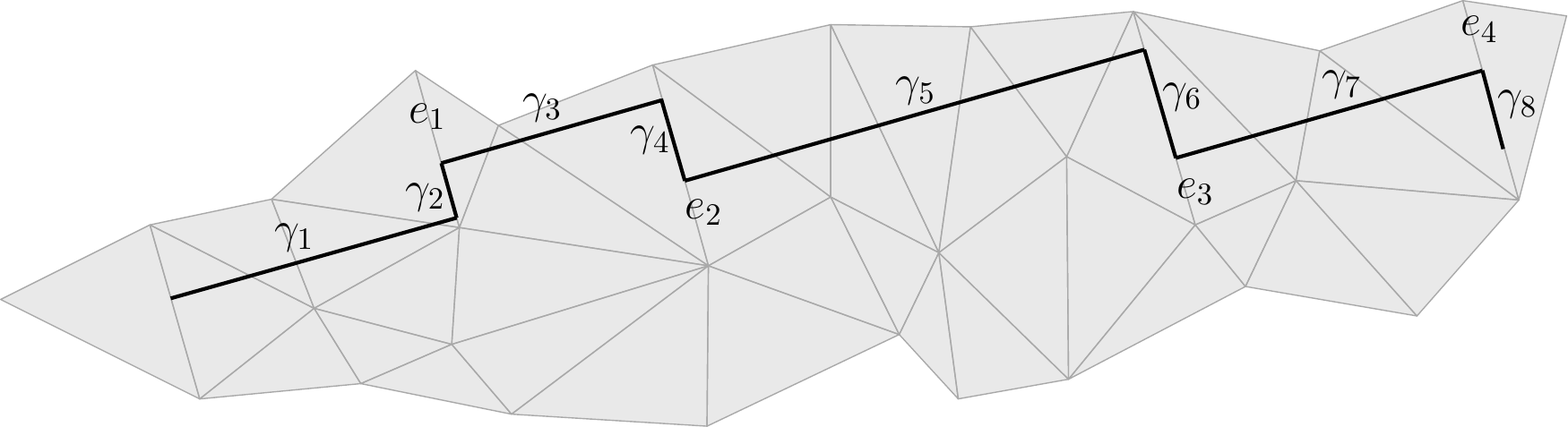}
	\caption{A sheared geodesic}
	\label{f:5}
\end{figure}

\begin{prop}
\label{prop:sheared}
Let $X$ be a piecewise Euclidean triangle complex that is~$\mathrm{CAT}(0)$. Let $\gamma$ in $X$ be a sheared geodesic. Then $\gamma$ is not a closed path.
\end{prop}

The proof will use the following two building blocks.

\begin{lemma}
\label{lem:edge}
Let $X$ be a piecewise Euclidean triangle complex that is $\mathrm{CAT}(0)$. Let $xy$ be a nontrivial geodesic in $X$ such that $y$ belongs to the interior of an edge $e$ of $X$ and $xy$ ends in a triangle $T$. Then for any $z$ in the interior of $e$, the geodesic $xz$ is nontrivial and ends in $T$.
\end{lemma}
\begin{proof}
We have $z\neq x$ since edges are geodesics and so in particular $x$ does not lie in $e$.
If for some $z$ in the interior of $e$ the geodesic $xz$ does not end in~$T$, then, since the geodesic $xz$ varies continuously with $z$, 
for some $z$ in the interior of $e$ the geodesic $xz$ ends in $e$. Denote by $e_1, e_2$ the two subedges into which such $z$ divides $e$. Suppose that $xz$ ends in $e_1$. Then the entire $e_1$ must lie in $xz$. Moreover, appending $xz$ by~$e_2$ we also obtain a geodesic (Theorem~\ref{thm:localisometry}). 
Since $y$ lies in $e=e_1\cdot e_2$, this shows that $xy$ ends in~$e$,
which is a contradiction.
\end{proof}

\begin{lemma}
\label{lem:notedge}
Let $X$ be a piecewise Euclidean triangle complex that is $\mathrm{CAT}(0)$. Let $xz$ be a nontrivial geodesic in $X$ such that $z$ belongs to the interior of an edge $e$ of $X$ and $xz$ ends in a triangle~$T$. Suppose that $y$ belongs to the interior of an edge $e'$ of $X$ and $zy$ is a nontrivial geodesic in $X$ that starts perpendicularly to $e$ in a triangle distinct from $T$ and ends perpendicularly to $e'$ in a triangle $T'$. Then the geodesic $xy$ is nontrivial and ends in $T'$.
\end{lemma}
\begin{proof}
Consider the geodesic triangle $xyz$. By our assumptions, its Alexandrov angle at $z$ is $>\frac{\pi}{2}$, and so in particular $x\neq y$,
and its Alexandrov angle at $y$ is $<\frac{\pi}{2}$. Since $zy$ ends in a triangle $T'$ perpendicularly to $e'$, we have that
$xy$ ends in $T'$, as desired.
\end{proof}

\begin{proof}[Proof of Proposition~\ref{prop:sheared}]
Let $\gamma=\gamma_1\cdot \gamma_2\cdots \gamma_{2k-1}\cdot\gamma_{2k}$ as in Definition~\ref{def:sheared}. For $i=1,\ldots, k,$ denote $\gamma_{2i}=y_iz_i$ and denote by $T_i$ the triangle in which $\gamma_{2i-1}$ ends. Let $x$ be the starting point of $\gamma_{1}$. We prove by induction on $i=1,\ldots, k,$ that
$x$ and $z_i$ are distinct and that the geodesic $xz_i$ ends in $T_i$. The proposition follows from this induction hypothesis applied with $i=k$.

For $i=1$ the induction hypothesis follows from Lemma~\ref{lem:edge}. Suppose now that we have established it for some $i=m<k$. Then by Lemma~\ref{lem:notedge} the geodesic $xy_{m+1}$ is nontrivial and ends in $T_{m+1}$. Thus by Lemma~\ref{lem:edge}, the
induction hypothesis holds for $i=m+1$.
\end{proof}

\section{Free}
\label{sec:free}

\begin{proof}[Proof of Proposition~\ref{prop:decompose}]
By Proposition~\ref{prop:unfolding}, we can assume that none of the vertex links of $X$ are unfoldable. 
Thus by Proposition~\ref{prop:rank1} and Lemma~\ref{lem:noteasy}(i) we can assume that $X$ is piecewise Euclidean and extrational. Let $Z\subset X$ be a thick {$G$-\sics} Note that each patch of $X$ either has no triangle in $Z$ or is contained in $Z$, in which case we call it a \emph{$Z$-patch}.

Since $X$ is rational, and $Z$ is a {$G$-\sics}, there is $q\in \N$ such that for each $Z$-patch $P$ and each vertex $v\in \partial P$, the length of $\lk_v^{\overline P}$ is a multiplicity of $\frac{\pi}{q}$. For each $Z$-patch $P$, we
define the homomorphism $\psi'=\psi'(P)\colon H_1(\overline P,\partial P)\to \R/\frac{\pi}{q}\Z$ in the same way as $\psi$, but replacing $\pi\Q$ by $\frac{\pi}{q}\Z$. We have $\psi=\psi'$ mod $\pi\Q$. Since $\psi$ is trivial, the image of $\psi'$ is contained in $\pi\Q/\frac{\pi}{q}\Z$. Since 
there are finitely many $G$-orbits of $Z$-patches, and since each $H_1(\overline P,\partial P)$ is finitely generated as a $\mathrm{Stab}(P)$-module, there is $q'\in \N$ such that the image of each~$\psi'$ is contained in $\frac{\pi}{q'}\Z/\frac{\pi}{q}\Z$. Consequently, for any $Z$-patch $P$, any geodesic $xy$ in~$\overline P$ disjoint from~$\partial P$, except at its endpoints, that is at angle $\in \frac{\pi}{q'}\Z$ from~$\partial P$ at $x$, is also at angle $\in \frac{\pi}{q'}\Z$ from $\partial P$ at $y$. Without loss of generality assume that $q'$ is even.

We need the following variant of the Liouville measure $\mu$ from \cite[\S3]{BB}. Let $S$ be the set of all the directions $\xi$ at an angle $\theta(\xi)\in \frac{\pi}{q'}\Z\cap (-\frac{\pi}{2},\frac{\pi}{2})$ from a direction normal to~$E$ in the links $\lk_x^Z$ for all the points $x\in Z$ that lie in the interior of an edge~$e$ of~$E$. The \emph{Liouville measure} $d\mu(\xi)$ on $S$ is given as $\cos \theta(\xi)dx$, where $dx$ is the volume element on $e$. Let $V\subset S$ be the full measure subset of $S$ of directions~$\xi$ such that each geodesic ray~$\gamma$ in~$Z$ with starting direction~$\xi$ is disjoint from $Z^0$. 
Let $F\colon V\to \mathcal P(V)$ be the map defined by $\eta\in F(\xi)$ for $\eta\in \lk_x^Z$ if there exists a geodesic $yz$ in $Z$ with starting direction $\xi$, intersecting~$E$ only in $y$ and $x$, and with $\eta$ being the direction at $x$ of $xz$. Since $G$ acts on $Z$ properly and cocompactly, we have that each $F(\xi)$ is finite. We can thus define a Markov chain with states $V$ and transition probabilities $\frac{1}{|F(\xi)|}$ from $\xi$ to each $\eta\in F(\xi)$. By (the calculation in) \cite[Prop~3.3]{BB}, the measure $\mu$ is stationary for this Markov chain. Thus the space $V^\Z$ can be equipped with Markov measure~$\mu^*$ invariant under the shift (see e.g.\ \cite[Ex~(8),~page~21]{Wa}). Since $Z$ is a {$G$-\sics}, the quotient $V^\Z/G$ by the diagonal action of $G$ is of finite measure. Note that the shift map descends to $V^\Z/G$ and is still measure preserving. 

Let $e$ be an edge of $Z$ lying in three distinct triangles $T_a,T_b,T_c$ of $Z$. Let $V^{ab}\subset V^\Z$ be the set of $(\xi_i)_i$ such that
\begin{itemize}
\item
we have $\xi_1\in \lk_x^{T_a}$ for $x\in e$, and $\xi_1$ is at angle $\frac{\pi}{2}$ from $e$, and
\item
the geodesic $yx$ from the definition of $\xi_1\in F(\xi_0)$ ends in $T_b$.
\end{itemize}
Note that $V^{ab}$ has positive Markov measure. Thus by the Poincar\'e recurrence (see e.g.\ \cite[Thm~1.4]{Wa}), there is $(\xi_i)_i\in V^{ab}$ and $j>0$ with $(\xi_{i-j})_i\in GV^{ab}$.
Consequently, there is a geodesic $\gamma^{ab}$ in $Z\setminus Z^0$ starting perpendicularly to~$e$ in~$T_a$ and ending perpendicularly to a translate $fe$ in $fT_b$, for some $f\in G$. 

Denote by $a, fb$ the endpoints of $\gamma^{ab}$. Let $I^{ab}$ be the domain of 
the isometric embedding $\gamma^{ab}\colon I^{ab}\to Z\setminus Z^0$. Analogously, there is a geodesic $\gamma^{ca}\colon I^{ca}\to Z\setminus Z^0$ starting perpendicularly to $e$ in $T_c$ and ending perpendicularly to a translate $ge$ in $gT_a$, with endpoints $c, a'$, for some $g\in G$. Finally, there is a geodesic $\gamma^{bc}\colon I^{bc}\to Z\setminus Z^0$ starting perpendicularly to $ge$ in $T_b$ and ending perpendicularly to a translate $f'ge$ in $f'gT_c$, with endpoints $b', f'c'$, for some $f'\in G$. Let $\gamma\colon I\to e$ be the shortest geodesic in $e$ containing all $a,b,c$ in its image (possibly $I$ is a single point), and let $\gamma'\colon I'\to ge$ be the shortest geodesic in $ge$ containing all $a',b',c'$ in its image.

Let $\Gamma$ be the metric graph obtained in the following way. We start from the disjoint union of the five intervals $I^{ab},I^{ca},I^{bc},I,I'$, and we identify (see Figure~\ref{f:3}):
\begin{itemize}
\item
points of $I$ and $I^{ab}$ mapping to $a$ under $\gamma$ and $\gamma^{ab}$,
\item
points of $I$ and $I^{ca}$ mapping to $c$ under $\gamma$ and $\gamma^{ca}$,
\item
points of $I'$ and $I^{ca}$ mapping to $a'$ under $\gamma'$ and $\gamma^{ca}$, and 
\item
points of $I'$ and $I^{bc}$ mapping to $b'$ under $\gamma'$ and $\gamma^{bc}$.
\end{itemize}
Note that $\Gamma$ admits the map $\varphi\colon \Gamma\to Z$ that is the quotient of $\gamma^{ab}\sqcup \gamma^{ca}\sqcup\gamma^{bc}\sqcup\gamma\sqcup\gamma'$. Let $s,t,s',t'$, be the points in $I,I^{ab},I',I^{bc}$ mapping under $\varphi$ to $b, fb, c',f'c'$, respectively.

\begin{figure}[h!]
	\centering
	\includegraphics[width=0.76\textwidth]{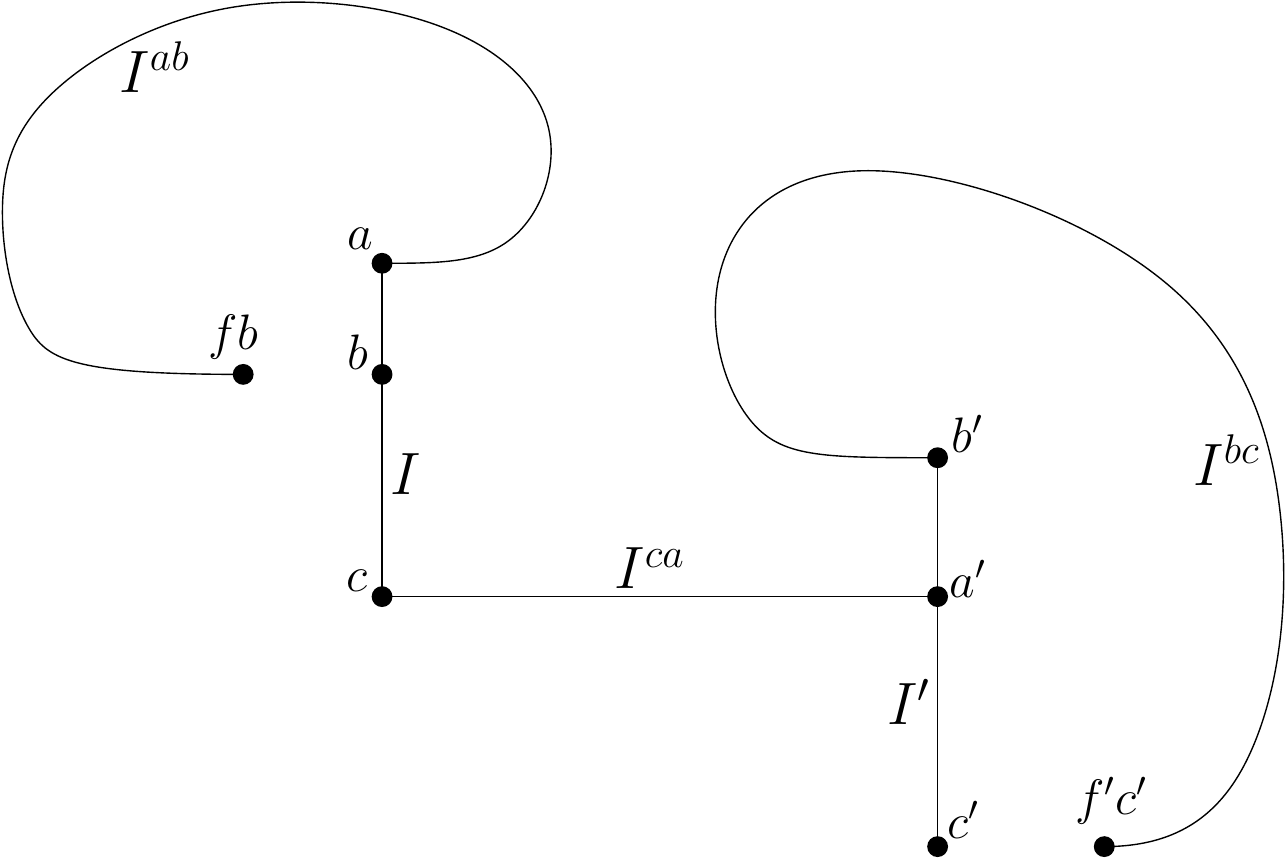}
	\caption{The graph $\Gamma$}
	\label{f:3}
\end{figure}

Let $F_2$ be the free group on two generators $h,h'$ and let $\widehat\Gamma$ be the quotient of the graph $F_2\times \Gamma$ (which is the disjoint union of $F_2$ copies of $\Gamma$) by the relations $w\times t\sim wh\times s, w\times t'\sim wh'\times s'$, for all $w\in F_2$. Note that $\widehat\Gamma$ is a tree with a free action of $F_2$. Let $\varphi_*\colon F_2 \to G$ be the homomorphism mapping $h,h'$ to $f,f'$, respectively. Then  
$\varphi$ extends to a $\varphi_*$-equivariant map $\widehat \varphi \colon \widehat \Gamma\to Z$ mapping each $w\times r\in F_2\times  \Gamma$ to $\varphi_*(w)\varphi(r)\in Z$. 

Let $w$ be a nontrivial element of $F_2$ and let $\R_w$ be the axis for $w$ in $\widehat\Gamma$. Pick $p\in \R_w$ an endpoint of a translate of one of $I^{ab},I^{ca},I^{bc}$ contained in~$\R_w$. Let $I_w\subset \R_w$ be the interval between $p$ and $wp$, and let $\gamma_w\colon I_w\to Z$ be the restriction of $\widehat \varphi$ to~$I_w$.  Since $T_a,T_b,T_c,$ were distinct, we have that $\gamma_w$ is a sheared geodesic. By Proposition~\ref{prop:sheared}, we have $\widehat \varphi(p)\neq \widehat \varphi(wp)=\varphi_*(w)\widehat \varphi(p)$, and consequently $\varphi_*(w)$ is nontrivial. Thus $\varphi_*$ is injective, and so $G$ contains a nonabelian free group.
\end{proof}

\end{document}